\documentclass[reqno, 11pt]{amsart}

\usepackage{enumitem}
\setenumerate[0]{label=(\roman*)}

\usepackage[T1]{fontenc}    
\usepackage[a4paper, hmargin={2.7cm,2.7cm},vmargin={3.3cm,3.3cm}]{geometry}
\usepackage{amssymb}   
\usepackage{amsthm}    
\usepackage{thmtools}  
\usepackage{mathtools} 
\usepackage{mathrsfs}  
\usepackage{commath}
\usepackage{bbm}
\usepackage[textwidth=20mm]{todonotes}

\RequirePackage[backend    = biber,
                sortcites  = true,
                giveninits = true,
                doi        = false,
                isbn       = false,
                url        = false,
                maxnames = 50,
                style      = numeric-comp]{biblatex}
\DeclareNameAlias{sortname}{family-given}
\DeclareNameAlias{default}{family-given}
\usepackage{varioref}
\usepackage{hyperref}
\usepackage[nameinlink, capitalize, noabbrev]{cleveref}

\declaretheorem[style = plain, numberwithin = section]{theorem}
\declaretheorem[style = plain,      sibling = theorem]{corollary}
\declaretheorem[style = plain,      sibling = theorem]{lemma}
\declaretheorem[style = plain,      sibling = theorem]{proposition}
\declaretheorem[style = definition, sibling = theorem]{definition}

\declaretheorem[style = definition, sibling = theorem]{remark}

\newtheorem{introthm}{Theorem}

\newcommand{\N}{\mathbb{N}}

\newcommand{\R}{\mathbb{R}}
\newcommand{\C}{\mathbb{C}}

\newcommand{\vN}{\mathrm{W}^*}

\DeclareFontFamily{U}{mathx}{}
\DeclareFontShape{U}{mathx}{m}{n}{<-> mathx10}{}
\DeclareSymbolFont{mathx}{U}{mathx}{m}{n}
\DeclareMathAccent{\widehat}{0}{mathx}{"70}
\DeclareMathAccent{\widecheck}{0}{mathx}{"71}

\def \brack#1#2{ {_\alpha\!}\langle #1,#2\rangle }
\def \dom{ \operatorname{dom} }

\def \M{ \mathcal{M} }
\def \meas#1{ L^0(#1) }
\def \supp{ \operatorname{supp} }

\title[A Duflo--Moore theorem for group actions on von Neumann algebras]{A Duflo--Moore theorem for ergodic group actions on semifinite von Neumann algebras}

\author{Ulrik Enstad}
\address{Department of Mathematics,
University of Oslo,
Moltke Moes vei 35,
0851 Oslo.}
\email{ubenstad@math.uio.no}

\author{Hannes Wendt}
\address{Department of Mathematics,
University of Oslo,
Moltke Moes vei 35,
0851 Oslo.}
\email{hhwendt@math.uio.no}

\begin{document}

\begin{abstract}
    We prove a generalization of the orthogonality relations of Duflo and Moore for ergodic, trace-preserving group actions on von Neumann algebras that are integrable in a suitable sense. We also obtain convolution inequalities that generalize both Young's inequality for convolution on locally compact groups and inequalities for operator-operator convolutions in Werner's quantum harmonic analysis.
\end{abstract}

\maketitle

\section{Introduction}

The orthogonality relations of Duflo and Moore \cite{DuMo1976} are fundamental to the representation theory of non-unimodular locally compact groups. They state that for a square-integrable unitary representation $\pi$ of a locally compact group $G$ on a Hilbert space $\mathcal{H}$, there exists a unique positive, self-adjoint, invertible operator $D$ on $\mathcal{H}$ such that
\begin{equation}
    \int_G \langle \xi, \pi_s \eta \rangle \overline{ \langle \xi', \pi_s \eta' \rangle } \dif{s} = \langle \xi, \xi' \rangle \overline{ \langle D^{-1/2} \eta , D^{-1/2} \eta' \rangle } , \quad \xi, \xi' \in \mathcal{H}, \; \eta , \eta' \in \dom(D^{-1/2}). \label{eq:duflo-moore}
\end{equation}
The operator $D$, known as the \emph{Duflo--Moore operator}\footnote{We follow the convention of \cite{DuMo1976} which has the advantage that if $G$ is compact and equipped with its normalized Haar measure, then $D = d_\pi I$ where $d_\pi$ is the dimension of $\pi$. We remark that two other common conventions for the definition of $D$ correspond to the inverse or the square root of our choice of $D$, respectively.} associated with $\pi$, may be unbounded; it is bounded exactly when $G$ is unimodular, in which case it is a constant multiple of the identity. The relations \eqref{eq:duflo-moore} also extend to projective, square-integrable representations, see e.g.\ \cite{An06}.

Duflo--Moore operators arise naturally in several branches of applied harmonic analysis. In wavelet theory, which centers around representations of the affine group and other semidirect products \cite{Fu05}, elements in $\dom(D^{-1/2})$ are referred to as admissible vectors. For projective square-integrable representations of $\mathbb{R}^{2d}$ (alternatively the Schrödinger representation of the Heisenberg group), which are essential to time-frequency analysis, the orthogonality relations \eqref{eq:duflo-moore} are known as Moyal's identity \cite[Theorem 3.2.1, Proposition 4.3.2]{Gr01}. In this case, the underlying group is unimodular, so all vectors are admissible.

Moyal's identity is the starting point of Werner's quantum harmonic analysis on phase space (QHA) \cite{We84}, which has seen connections to time-frequency analysis in the last few years (see e.g.\ \cite{LuSk18,LuSk19,LuSk20,LuSk21,BeBeLu22}). An affine version of QHA was recently constructed in \cite{BeBeLu22} with analogous connections to wavelet theory. In \cite{Ha23,FuGa25} a foundation for QHA in the setting of a general square-integrable representation $\pi$ was developed, which rests on the orthogonality relations of Duflo and Moore \eqref{eq:duflo-moore}. One of the core concepts in all these variants of QHA is the operator-operator convolution, which associates to two operators $A$ and $B$ on $\mathcal{H}$ the scalar-valued function on $G$ given by $s \mapsto \mathrm{tr}(A \alpha_s(B))$, where $\alpha_s(B) = \pi_s B \pi_s^*$ is termed the \emph{shift} of the operator $B$ by $s \in G$. (Strictly speaking, the operator-operator convolution in Werner's work \cite{We84} is defined using a parity operator, see \Cref{rmk:parity}.) Extending the notion of an admissible vector, an operator $B$ is called admissible if, roughly speaking, $D^{-1/2} B D^{-1/2}$ defines a trace class operator, where $D$ is the associated Duflo--Moore operator. The generalization of \eqref{eq:duflo-moore} in QHA then states that
\begin{equation}
    \int_G \mathrm{tr}(A \alpha_s(B)) \dif{s} = \mathrm{tr}(A) \mathrm{tr}(D^{-1/2} B D^{-1/2})  \label{eq:duflo-operators}
\end{equation}
whenever $A$ is trace class and $B$ is admissible. There is an interplay between the operator-operator convolution and a compatible function-operator convolution that shares many of the properties of classical convolution on groups. For instance, they satisfy versions of Young's inequality (see \cite[Proposition 4.2]{LuSk18}, \cite[Proposition 4.18]{BeBeLu22}, and \cite[Proposition 4.13]{Ha23}).

The goal of this article is to show that both classical convolution of scalar-valued functions on a locally compact group as well as the operator-operator convolution in QHA can be treated simultaneously in a general framework, namely that of an ergodic, trace-preserving action of a locally compact group on a semifinite von Neumann algebra that is integrable in a suitable sense. Our first main theorem (\Cref{thm:intro}) is a direct generalization of \eqref{eq:duflo-operators} to this setting, while our second main theorem (\Cref{thm:intro2}) is a convolution inequality that generalizes both the classical Young's convolution inequality on a locally compact group as well as convolution inequalities in quantum harmonic analysis.

To state our main results, let $\M$ be a von Neumann algebra equipped with an ergodic action $\alpha$ of a locally compact group $G$. We assume a normal, semifinite, faithful trace $\tau$ on $\M$ which is invariant with respect to the action. The natural Banach spaces of operators in this setting are the noncommutative $L^p$-spaces $L^p(\M)$ associated with $\tau$, consisting of (possibly unbounded) operators affiliated with $\M$ defined in terms of the norms $\| x \|_p = \tau(|x|^p)^{1/p}$ for $x \in \M$. When defining an analogue of the operator-operator convolution, we take a slightly different viewpoint than \cite{Ha23,BeBeLu22}. The definitions in these papers do not involve a parity operator and hence do not offer true generalizations of the operator-operator convolution in \cite{We84} (see \Cref{rmk:parity}). Instead we introduce a bracket product $\brack{\cdot}{\cdot}$ taking suitable elements $x$ and $y$ and returning a complex-valued function $\brack{x}{y}$ on $G$. Our notation is inspired by Hilbert C*-modules and indeed $\brack{\cdot}{\cdot}$ is linear with respect to a function-operator convolution in the first argument and conjugate-linear in the second argument, although we do not emphasize this viewpoint in the present paper. When $x \in L^1(\M)$ and $y \in \M$, we define
\[
    \brack{x}{y}(s) = \tau( x \alpha_s(y^*) ), \quad s \in G.
\]
The appropriate generalization of square-integrability in this setting is that there exist nonzero elements $x, y \in \M_+$ with $\tau(x)< \infty$ such that $\brack{x}{y} \in L^1(G)$. In this case, we say that the action of $G$ on $\M$ is \emph{$\tau$-integrable}, see \Cref{def:admissible-action}.

Our first main theorem shows that this bracket is an integrable function on $G$ when $y$ is admissible in a suitable sense generalizing the notion of admissibility from \cite{BeBeLu22,Ha23}, see \Cref{subsec:admissibility}. This admissibility is governed by a Duflo--Moore operator that arises naturally in this setting.

\begin{introthm}\label{thm:intro}
    Let $\M$ be a von Neumann algebra equipped with a normal, semifinite, faithful, tracial weight $\tau$. Suppose that $\M$ is equipped with an ergodic, $\tau$-preserving and $\tau$-integrable action of a locally compact group $G$. Then there exists a unique positive invertible operator $D$ affiliated with $\M$ with the following property: 
    
    For any $x \in L^1(\M)$ and any admissible $y \in \M$, or more generally $y \in L^1(\M,D^{-1})$\footnote{See \Cref{subsec:admissibility} for the definition of the Banach space $L^1(\M,D^{-1})$.}, the function $\brack{x}{y}$ is integrable, with 
    \[ 
        \| \brack{x}{y} \|_1 \leq \| x \|_1 \| D^{-1/2} y D^{-1/2} \|_1, 
    \] 
    and 
    \[ 
        \int_G \brack{x}{y}(s) \dif{s} = \tau(x) \overline{\tau(D^{-1/2} y D^{-1/2})}.
    \]
    
    Furthermore, $D$ satisfies the following semi-invariance relation: 
    \[ 
        \alpha_s(D) = \Delta(s)^{-1} D, \quad s \in G, 
    \] where $\Delta$ denotes the modular function of $G$. In particular, $D$ is bounded if and only if $G$ is unimodular, in which case $D$ is a constant multiple of the identity.
\end{introthm}

The above theorem unifies the Duflo--Moore Theorem of Werner's quantum harmonic analysis (see \Cref{subsec:square-integrable}) and $L^1$-integrability for ordinary convolution on a locally compact group (see \Cref{subsec:classical-convolution}). It is a combination of \Cref{thm:duflo-moore} and \Cref{cor:first_main_theorem}.

Our second main theorem (see \Cref{cor:second_main_theorem}) is the following inequality, which contains Young's convolution inequality for functions on a locally compact group and the convolution inequalities in quantum harmonic analysis as special cases.

\begin{introthm}\label{thm:intro2}
    In the setting of \Cref{thm:intro}, let $p,q,r \geq 1$ satisfy $1/p + 1/q = 1 + 1/r$. Let $x \in L^p(\M)$, $y \in L^q(\M)$, and suppose that $y$ commutes with $D$. Then $\brack{x}{ D^{1/r}y } \in L^r(G)$, with 
    \[
        \| \brack{x}{ D^{1/r}y } \|_r \leq \| x \|_p \| y \|_q. 
    \]
\end{introthm}

Our framework encompasses other examples as well, as shown in \Cref{sec:examples}. For instance, the Fourier inversion formula on a locally compact abelian group can be seen as a consequence of \Cref{thm:intro}, see \Cref{subsec:fourier-inversion}.

\subsection*{Related work}

Since the first version of this preprint was uploaded to the arXiv, the authors were made aware of related work of Arhancet \cite[Section 7.2, 7.3]{Ar25}. Here, the setting is a locally compact quantum group $\mathbb{G}$ acting ergodically and trace-preservingly on a von Neumann algebra $\M$. A version of Werner's operator-operator convolution is defined for pairs of elements $x \in L^\infty(\M)$ and $y \in L^1(\M)$, and \cite[Proposition 7.5]{Ar25} corresponds to \Cref{thm:intro2} for $r = \infty$, $p = 1$ and $q = \infty$ (note that when $r = \infty$, the Duflo--Moore operator makes no appearance in Young's convolution inequality). Thus, there is some overlap between \cite{Ar25} and the main results of the present paper.

\subsection*{Acknowledgements}

The authors are indebted to Franz Luef, Eduard Ortega, and Makoto Yamashita for valuable discussions, and to the anonymous referee for many useful comments and suggestions.

\subsection*{Assumptions}

Throughout the text, $\M$ denotes a von Neumann algebra. We will assume $\M$ to be equipped with a normal, semifinite, faithful, tracial weight which we denote by $\tau$, and we represent $\M$ as bounded linear operators on the corresponding GNS Hilbert space $L^2(\M)$ (see \Cref{subsec:weights} for definitions).

We also fix a locally compact group $G$ with left Haar measure $m$ and modular function $\Delta$. Integration of a measurable function $f$ on $G$ with respect to $m$ is denoted by $\int_G f(s) \dif{s}$. We assume that $G$ acts ergodically and $\tau$-preservingly on $\M$, see \Cref{sec:group-actions}. We will assume that the action is \emph{$\tau$-integrable}, which means that there exist nonzero positive elements $x,y \in \M_+$ such that 
\[
    \int_G \tau(x^{1/2} \alpha_{s}(y) x^{1/2}) \dif{s} < \infty. 
\]
See \Cref{def:admissible-action}.

\section{Admissibility with respect to an affiliated operator}

The goal of this section is to cover preliminaries on weights on von Neumann algebras and to introduce the notion of $A$-admissibility, where $A$ is a positive operator affiliated with $\M$. This notion will later be applied to the case where $A$ is the inverse of the Duflo--Moore operator associated with an ergodic, $\tau$-preserving, $\tau$-integrable group action on $\M$.

\subsection{Weights}\label{subsec:weights}

For a detailed introduction to weights on von Neumann algebras we refer to \cite[Chapter VII]{Ta03}. A \emph{weight} on $\M$ is a map $\phi \colon \M_+ \to [0,\infty]$ such that 
\[ 
    \phi( \alpha x + \beta y) = \alpha \phi(x) + \beta \phi(y) 
    \quad \text{ for all } x, y \in \M_+ \text{ and } \alpha, \beta \in [0,\infty). 
\] Here we use the convention that $0 \cdot \infty = 0$.

We call $\phi$:
\begin{enumerate}
    \item \emph{normal} if $\phi( \sup_{i \in I} x_i) = \sup_{i \in I} \phi(x_i)$ for bounded, increasing nets $(x_i)_{i \in I}$ in $\M_+$;
    \item \emph{semifinite} if the set $\{ x \in \M_+ : \phi(x) < \infty \}$ generates $\M$ as a von Neumann algebra;
    \item \emph{faithful} if $\phi(x) = 0$ implies $x = 0$ for all $x \in \M_+$; and
    \item \emph{tracial} if $\phi(x^* x) = \phi(x x^*)$ for all $x \in \M$.
\end{enumerate}
For a weight $\phi$ the set
\begin{equation}
    \mathfrak{n}_\phi = \{ x \in \M : \phi(x^*x) < \infty \} \label{eq:weight-ideal}
\end{equation}
is a left ideal in $\M$, hence by \cite[Proposition II.3.12]{Ta02} its ultraweak closure is of the form $\mathcal{M}p_\phi$ for a uniquely determined projection $p_\phi \in \mathcal{M}$. The semifiniteness of $\phi$ is equivalent to $p_\phi = 1$.

We set $\mathfrak{n}_\phi^* = \{ x^* : x \in \mathfrak{n}_\phi \}$ and
\[ \mathfrak{m}_\phi = \mathrm{span} \{ x^*y : x,y \in \mathfrak{n}_\phi \} = \mathrm{span} \{ x y^* : x,y \in \mathfrak{n}_{\phi}^* \} . \]
The set $\mathfrak{m}_\phi$ is a hereditary $*$-subalgebra of $\M$ and every element of $\mathfrak{m}_\phi$ is a linear combination of four positive elements $x \in \M_+$ such that $\phi(x) < \infty$, see \cite[p.\ 41, Lemma 1.2]{Ta03}.

Similarly
\begin{equation}
    \mathcal{N}_\phi = \{ x \in \M : \phi(x^*x) = 0 \} \label{eq:null_ideal}
\end{equation}
is a left ideal in $\mathcal{M}$ and the \emph{support projection} of $\phi$ is the unique projection $\supp \phi$ such that the ultraweak closure of $\mathcal{N}_\phi$ equals $\mathcal{M} (1-\supp \phi)$. It satisfies \[
    \phi(x^* x) = 0 \iff x \supp \phi = 0, \quad x \in \M.
\] In particular, $\phi$ is faithful if and only if $\supp \phi = 1$.

As mentioned, $\tau$ will denote a normal, semifinite, faithful, tracial weight on $\M$. On the set $\mathfrak{n}_\tau$ we define the norm $\| x \|_2 = \tau(x^* x)^{1/2}$ and denote by $L^2(\M)$ the corresponding Hilbert space completion. We will represent $\M$ as a von Neumann subalgebra of $\mathcal{B}(L^2(\M))$ by identifying $x \in \M$ with the operator on $L^2(\M)$ given by sending $y \in \mathfrak{n}_\tau$ to $x y \in \mathfrak{n}_\tau$.

\subsection{The Radon--Nikodym Theorem}

Recall that a closed, densely defined operator $A$ on $L^2(\M)$ is said to be \emph{affiliated with $\M$} if it commutes with every unitary operator in the commutant of $\M$ inside $\mathcal{B}(L^2(\M))$. Every affiliated operator has a polar decomposition $A = u |A|$ where $u \in \M$ is a partial isometry and $|A|$ is a positive operator affiliated with $\M$. One calls $A$ \emph{invertible} if its kernel is zero, and in this case $A$ has a closed, densely defined inverse $A^{-1}$ which is also affiliated with $\M$.

For $y \in \M_+$ one may define a weight $\tau_y$ on $\M$ by 
\[ 
    \tau_y(x) = \tau(y^{1/2} x y^{1/2}), \quad x \in \M_+. 
\]

This definition may be extended to positive operators $A$ affiliated with $\M$ as follows. We denote
\[
    A_\epsilon = A (1 + \epsilon A)^{-1} \in \M_+, \quad \epsilon > 0.
\]
The net $(A_\epsilon)_{\epsilon > 0}$ is increasing in $\M_+$ as $\epsilon \to 0$, and
\[ 
    \tau_A(x) = \lim_{\epsilon \to 0} \tau_{A_\epsilon}(x), \quad x \in \M_+, 
\]
defines a normal and semifinite weight on $\M$, see \cite[p.\ 103, Lemma 2.8]{Ta03}.
The following Radon--Nikodym theorem will be instrumental to us, see \cite[p.\ 122, Theorem 3.14]{Ta03}.

\begin{theorem}\label{thm:radon-nikodym}
    For every normal, semifinite weight $\phi$ on $\M$ there exists a unique positive operator $A$ affiliated with $\M$ such that \[
        \phi = \tau_A. 
    \] The operator $A$ is called the \emph{Radon--Nikodym} derivative of $\phi$. Moreover, $A$ is invertible if and only if $\phi$ is faithful.
\end{theorem}

For a weight $\phi$ on $\M$ and a positive operator $B$ affiliated with $\M$ we also set \[ 
    \phi(B) = \sup_{\epsilon > 0} \phi(B_\epsilon) 
    \in [0, \infty]. 
\]

The following lemma will be applied frequently.

\begin{lemma}\label{lem:increasing_lemma}
Let $A$ be a positive self-adjoint operator affiliated with $\M$. Let $(f_n)_{n \in \N}$ be a pointwise increasing sequence bounded Borel functions on $\sigma(A)$ that converges pointwise to a bounded Borel function $f$ on $\sigma(A)$. Then for all $x \in \M$, the sequence $(\tau_{f_n(A)}(x))_{n \in \N}$ is increasing and converges to $\tau_{f(A)}(x)$.
\end{lemma}

\begin{proof}
By the functional calculus for unbounded operators, see \cite[Proposition 4.12 (v), Section 5.1, 5.2]{Sc12} we have that $(f_n(A))_{n \in \N}$ is an increasing sequence of operators that converges in the strong operator topology to $f(A)$. From \cite[Proposition 4.1]{PeTa73}, it follows that $\tau_{f_n(A)}(x)$ increases and converges to $\tau_{f(A)}(x)$ for all $x \in \M_+$.
\end{proof}

\subsection{Noncommutative integration}

An affiliated operator $A$ is called \emph{$\tau$-measurable} if there exists a projection $p \in \M \subseteq \mathcal{B}(L^2(\M))$ such that $p(L^2(\M)) \subseteq \dom(A)$, $\|A p\| < \infty$, and $\tau(1-p) < \infty$. We denote by $\meas{\M}$ the set of all $\tau$-measurable operators, and by $\meas{\M}_+$ the subset of positive operators. The set $\meas{\M}$ is a $*$-algebra with respect to the closure of the usual algebraic operations, cf.\ \cite[Chapter IX]{Ta03}.

The $L^p$-spaces associated with $\M$ are defined as \[ 
    L^p(\M) = \{ x \in \meas\M :  \|x\|_p < \infty \}, 
\]
where $\|x\|_p = \tau(|x|^p)^{1/p}$ for $x \in \meas\M_+$ and $1 \leq p < \infty$. One also sets $L^\infty(\M) = \M$. These are Banach spaces and the norms satisfy Hölder's inequality, that is,
\begin{align}
    \| x y \|_r \leq \| x \|_p \| y \|_q
\end{align}
whenever $x \in L^p(\M)$ and $y \in L^q(\M)$ with $1/p + 1/q = 1/r$ for $1 \leq p,q,r \leq \infty$.

\subsection{Admissibility}\label{subsec:admissibility}

In this subsection we fix a positive, self-adjoint and invertible operator $A$ affiliated with $\M$.

\begin{definition}\label{def:A-admissible}
    We say that an element $y \in \M$ is \emph{$A$-admissible} if it is an element of $\mathfrak{m}_{\tau_A}$, that is, if there exist $a_i, b_i \in \M$, $1 \leq i \leq n$, with $\tau_A(a_i a_i^*) < \infty$, $\tau_A(b_i b_i^*) < \infty$ such that
    \[ y = \sum_{i=1}^n a_i b_i^*. \]
\end{definition}

Recall that by definition $a \in \mathfrak{n}_{\tau_A}^*$ if and only if $\tau_A(aa^*) < \infty$.

\begin{lemma}\label{lem:A-compatible}
    Let $x \in \mathfrak n_{\tau_A}^*$. Then $(A_\epsilon^{1/2} x)_{\epsilon > 0}$ is a net of elements in $\mathfrak n_\tau$ that converges in $L^2(\M)$ as $\epsilon \to 0$.
\end{lemma}

\begin{proof}
    First, we check that $A_\epsilon^{1/2} x \in \mathfrak n_\tau$ for all $\epsilon > 0$: \[
        \tau( (A_\epsilon^{1/2} x) (A_\epsilon^{1/2} x)^* )
        = \tau_{A_\epsilon}(x x^*)
        \leq \tau_A(x x^*) < \infty.
    \]

Now let $\epsilon, \epsilon' > 0$. Note that if $\epsilon \leq \epsilon'$, then $A_{\epsilon} \geq A_{\epsilon'}$, so
\[ A_\epsilon^{1/2} A_{\epsilon'}^{1/2} = A_{\epsilon'}^{1/4} A_{\epsilon}^{1/2} A_{\epsilon'}^{1/4} \geq A_{\epsilon'}^{1/4} A_{\epsilon'}^{1/2} A_{\epsilon'}^{1/4} = A_{\epsilon'}. \]
Reversing the roles of $\epsilon$ and $\epsilon'$ shows that in either case we have
\[ A_{\epsilon}^{1/2} A_{\epsilon'}^{1/2} \geq A_{\max \{\epsilon, \epsilon' \}} . \]
Applying this, we see that \begin{align*}
    \tau(| A_\epsilon^{1/2} x - A_{\epsilon'}^{1/2} x |^2) 
    &= \tau(|A_\epsilon^{1/2} x|^2) + \tau(|A_{\epsilon'}^{1/2} x|^2) - \tau(x^* \Big( A_\epsilon^{1/2} A_{\epsilon'}^{1/2} + A_{\epsilon'}^{1/2} A_\epsilon^{1/2} \Big) x) \\
    &\leq 2 \tau_{A}(x x^*) - 2\tau(x^* A_{\max\{\epsilon, \epsilon'\}} x) \\
    &= 2 \big( \tau_A(xx^*) - \tau_{A_{\max\{\epsilon, \epsilon'\}} }(xx^*) \big) .
\end{align*}
Since $\tau_{A_\epsilon}(xx^*) \to \tau_A(xx^*) < \infty$ as $\epsilon \to 0$, the above estimate shows that the net is Cauchy in $L^2(\M)$. This finishes the proof.
\end{proof}

\begin{proposition}\label{prop:defn-L1MA-as-banachspace}
    If $y \in \M$ is $A$-admissible, then the net $(A_\epsilon^{1/2} y A_\epsilon^{1/2})_{\epsilon > 0}$ converges in $L^1(\M)$ as $\epsilon \to 0$, and we denote its limit by
    \[ A^{1/2} y A^{1/2} = \lim_{\epsilon \to 0} A_\epsilon^{1/2} y A_\epsilon^{1/2} \; \in L^1(\M). \]
    Moreover, the function \[ 
        \| y \|_{1, A} = \tau( | A^{1/2} y A^{1/2} | ), \quad y \in \mathfrak{m}_{\tau_A},
    \]
    defines a norm on the set of $A$-admissible elements. We denote resulting Banach space completion by $L^1(\M, A)$.
\end{proposition}

\begin{proof}
    Write $y = \sum_{i=1}^n a_i b_i^*$ with $a_i, b_i \in \mathfrak{n}_{\tau_A}^*$. Then
    \[ A_\epsilon^{1/2} y A_\epsilon^{1/2} = \sum_{i=1}^n (A_\epsilon^{1/2} a_i) (A_\epsilon^{1/2} b_i)^*. \]
    By \Cref{lem:A-compatible}, the nets $(A_\epsilon^{1/2} a_i)_{\epsilon > 0}$ and $(A_\epsilon^{1/2} b_i)_{\epsilon > 0}$ converge in $L^2(\M)$, so it follows from the above identity that $(A_\epsilon^{1/2} y A_\epsilon^{1/2})_{\epsilon > 0}$ converges in $L^1(\M)$.

    Since the assignment $\mathfrak{m}_{\tau_A} \to L^1(\M)$, $y \mapsto D^{1/2} y D^{1/2}$ is linear, $\| \cdot \|_{1, A}$ evidently satisfies the triangle inequality and absolute homogeneity.

Suppose that $y = \sum_{i=1}^n a_i b_i^*$ with $\tau_A(a_i b_i^*), \tau_A(b_i b_i^*)< \infty$ satisfies $\| y \|_{1, A} = 0$, so that $A^{1/2} y A^{1/2} = 0$. Letting $e_r$ be the spectral projection of $A$ onto the interval $[1/r,r]$, the operators $A^{-1/2} e_r = e_r A^{-1/2}$ are bounded for each $r > 0$. The estimate \[ 
    \| e_r - (1 + \epsilon A)^{-1} e_r \| 
    = \sup_{\lambda \in \sigma(A) \cap [1/r,r]} \Big| 1 - \frac{1}{1+ \epsilon \lambda^{-1}} \Big| 
    \leq \frac{\epsilon r}{1 + \epsilon r}. 
\]
shows that $(1 + \epsilon A^{-1})e_r \to e_r$ in the norm of $\M$ as $\epsilon \to 0$. Consequently \[ 
    (e_r A^{-1/2}) (A^{1/2} a_i) 
    = \lim_{\epsilon \to 0} \, e_r A^{-1/2} A_\epsilon^{1/2} a_i
    = \lim_{\epsilon \to 0} \, (1 + \epsilon A)^{-1} e_r a_i 
    = e_r a_i. 
\]
A similar computation shows that $(A^{1/2} b_i)^* (A^{-1/2} e_r) = b_i^* e_r$. Hence \[
    0 = (e_r A^{-1/2}) ( A^{1/2} y A^{1/2} ) (A^{-1/2} e_r) 
    = \sum_{i=1}^n e_r a_i b_i^* e_r = e_r y e_r. 
\]
Since $e_r \to I$ strongly as $r \to \infty$, we conclude that $y = 0$.
\end{proof}

\section{Group actions}\label{sec:group-actions}

As mentioned in the introduction, $G$ denotes a locally compact group with a left Haar measure $m$. We denote by $\Delta \colon G \to (0,\infty)$ the modular function of $G$, so that \[ 
    m(B s) = \Delta(s) m(B), 
    \quad B \subseteq G \text{ Borel}, \; s \in G. 
\]
We also fix an action $\alpha$ of $G$ on $\M$, by which we understand a group homomorphism from $G$ into the group $\operatorname{Aut}(\M)$ of $*$-automorphisms of the von Neumann algebra $\M$, which satisfies that for every $x \in \M$ the map $G \to \M$, $s \mapsto \alpha_s(x)$ is continuous with respect to the ultraweak topology on $\M$. We assume the action to be \emph{ergodic}, that is, the fixed point subalgebra \[ 
    \M^G = \{ x \in \M : \alpha_s(x) = x \text{ for all } s \in G \} 
\]
coincides with the scalar multiples of the identity of $\M$.

We also assume that the action of $G$ on $\M$ is \emph{$\tau$-preserving}, that is,
\begin{align*}
    \tau(\alpha_s(x)) = \tau(x), 
    \quad s \in G, \; x \in \M. 
\end{align*}
This induces an isometric action of the group $G$ on the Banach space $L^p(\M)$ for any $1 \leq p \leq \infty$, still denoted by $\alpha$.

The following proposition extends the action of $G$ on $\M$ to the space of operators affiliated with $\M$. Recall that we represent all operators on $L^2(\M)$.

\begin{proposition}\label{prop:action-affiliated}
    Let $A$ be an operator affiliated with $\M$ and let $s \in G$. Then there exists a unique operator $\alpha_s(A)$ affiliated with $\M$ such that $\dom(\alpha_s(A)) = \alpha_s(\dom(A))$ and such that
    \begin{equation}
        \alpha_s(A)\xi = \alpha_s(A (\alpha_{s^{-1}}(\xi))), 
        \quad \xi \in L^2(M) . \label{eq:action-affiliated}
    \end{equation}
\end{proposition}

\begin{proof}
    It is clear that \eqref{eq:action-affiliated} uniquely defines a linear operator $\alpha_s(A)$ on $L^2(\M)$ with the claimed domain. It remains to show that it is affiliated with $\M$. It is densely defined since $\dom(A)$ is dense in $L^2(\M)$ and $\xi \mapsto \alpha_s(\xi)$ is a homeomorphism. Furthermore, suppose $\xi_n \to \xi$ in $\dom(\alpha_s(A))$ and $\alpha_s(A) \xi_n \to \eta$ in $L^2(\M)$. Then $\alpha_{s^{-1}}(\xi_n) \to \alpha_{s^{-1}}(\xi)$ in $\dom(A)$ and $A (\alpha_{s^{-1}}(\xi_n)) \to \alpha_{s^{-1}}(\eta)$, so $A \alpha_{s^{-1}}(\xi) = \alpha_{s^{-1}}(\eta)$ since $A$ is closed. It follows that $\alpha_s(A) \xi = \alpha_s(A (\alpha_{s^{-1}}(\xi))) = \eta$, which shows that $\alpha_s(A)$ is closed. Finally, suppose $u$ is any linear operator in the commutant of $\M$ on $L^2(\M)$. Then the equation 
    \[
        \alpha_{s^{-1}}(u) x 
        = \alpha_{s^{-1}}( u \alpha_s(x) ) 
        = \alpha_{s^{-1}}( \alpha_s(x) u ) 
        = x \alpha_{s^{-1}}(u)
    \]
    shows that $\alpha_{s^{-1}}(u)$ is also in the commutant of $\M$ on $L^2(\M)$. Using that $A$ is affiliated we obtain 
    \[ 
        \alpha_s(A) u \xi 
        = \alpha_s(A \alpha_{s^{-1}}(u \xi)) 
        = \alpha_s(A \alpha_{s^{-1}}(u) (\alpha_{s^{-1}}(\xi))) 
        = \alpha_s(\alpha_{s^{-1}}(u) A (\alpha_{s^{-1}}(\xi))) 
        = u \alpha_s(A) \xi 
    \]
    for every $\xi \in L^2(\M)$. Since $u$ was arbitrary, this shows that $\alpha_s(A)$ is affiliated with $\M$.
\end{proof}

\begin{lemma}\label{lemma:epsilon-action}
    For all operators $A$ affiliated with $\M$ and all $s \in G$ we have \[
        (\alpha_s(A))_\epsilon = \alpha_s(A_\epsilon), \quad \epsilon > 0.
    \]
\end{lemma}
\begin{proof}
    Let $\epsilon > 0$ and $a = (1+\epsilon A)^{-1} \in \M$. Since $\alpha_s(A a) = \alpha_s(A)\alpha_s(a)$, one observes \[
        1 = \alpha_s(1) = \alpha_s((1+\epsilon A) a) = \alpha_s(a + \epsilon A a) = \alpha_s(a) + \epsilon \alpha_s(A) \alpha_s(a)  = (1 + \epsilon \alpha_s(A)) \alpha_s(a).
    \] Similarly $\alpha_s(a)(1 + \epsilon \alpha_s(A)) = 1$, so one finds $\alpha_s(a) = (1 + \epsilon \alpha_s(A))^{-1}$. Hence \[
        \alpha_s(A_\epsilon) = \alpha_s(A a) = \alpha_s(A) \alpha_s(a) = \alpha_s(A) (1 + \epsilon \alpha_s(A))^{-1} = (\alpha_s(A))_\epsilon.
    \]
\end{proof}

\begin{proposition}\label{prop:ergodic-affiliated}
    Let $A$ be a closed, densely defined operator affiliated with $\M$ such that $\alpha_s(A) = A$ for every $s \in G$. Then $A = \lambda 1$ for some $\lambda \in \C$.
\end{proposition}
\begin{proof}
    First, consider a positive affiliated operator $A$. Denote by $f \mapsto f(A)$ the spectral calculus of $A$, and by $\mathbbm{1}_B$ the indicator functions of Borel subsets $B \subseteq \R$. Then $P_B(A) = \mathbbm{1}_B(A) \in \M$ yield the spectral projections of $A$. For all $\xi,\eta \in \dom(A)$ we denoted by $\mu_{\xi,\eta}^A$ the spectral measures on $\R$ which are defined by $\mu_{\xi,\eta}^A(B) = \langle P_B(A)\xi, \eta \rangle$. The spectral theorem states that
    \begin{equation}\label{eq:spectral-theorem}
        \langle A\xi, \eta \rangle = \int_0^\infty t \dif{\mu_{\xi,\eta}^A(t)}.
    \end{equation}
    
    Fix any Borel subset $B \subseteq \R$ and $s \in G$. As the automorphism $\alpha_s$ is compatible with the functional calculus, we find $\alpha_s(P_B(A)) = P_B(\alpha_s(A))$. As $A$ is assumed to be invariant, it follows that $\alpha_s(P_B(A)) = P_B(A)$. As $s \in G$ was arbitrary and the action ergodic, there exists a scalar $\mu(B) \geq 0$ for which $P_B(A) = \mu(B) 1$. This shows $\mu_{\xi,\eta}^A(B) = \mu(B) \langle \xi,\eta\rangle$ for all $\xi,\eta \in \dom(A)$. 
    
    The properties of the spectral projections $P_B(A)$ imply that $\mu$ is a positive measure on $\R$. With $c = \int_0^\infty t \dif\mu(t) \in [0,\infty]$ the equality \eqref{eq:spectral-theorem} reduces to $\langle A\xi, \eta \rangle = c \langle \xi,\eta\rangle$ for all $\xi,\eta \in \dom(A)$, showing that $c < \infty$ and that $A$ coincides with $c 1 \in \M$ on its dense domain of definition. From \cite[Lemma IX.2.7]{Ta03} it follows that $A$ must have been everywhere defined and $A = c 1$. 
    
    For a general affiliated operator $A$, consider the polar decomposition $A = u |A|$ with partial isometry $u \in \M$ and positive, affiliated operator $|A|$. Let $s \in G$. Then $A = \alpha_s(A) = \alpha_s(u) \alpha_s(|A|)$. In this decomposition $\alpha_s(|A|)$ is a positive operator and $\alpha_s(u) \in \M$ a partial isometry with initial space being the range of $\alpha_s(|A|)$. By the uniqueness of the polar decomposition, $|A| = \alpha_s(|A|)$ and $u = \alpha_s(u) \in \M$. As $s \in G$ was arbitrary and the action ergodic, $u = \lambda 1$ for some $\lambda \in \C$ with $|\lambda| = 1$, and by the above $|A| = c 1$ for some $c \geq 0$. Then $A = u |A| = \lambda c 1$ shows the claim.
\end{proof}

For a weight $\phi$ on $\M$ and an element $s \in G$ we define the weight $\alpha_s(\phi)$ on $\M$ via \[
    \alpha_s(\phi)(x) = \phi(\alpha_{s^{-1}}(x)), \quad x \in \M.
\]
Let $\chi$ be a character on $G$, that is, a continuous group homomorphism from $G$ into the group of positive real numbers under multiplication. A weight $\phi$ is called \emph{semi-invariant} with respect to $\chi$ if \[ 
    \alpha_s(\phi) = \chi(s)\phi, \quad s \in G. 
\]
We say that $\phi$ is \emph{invariant} if it is semi-invariant with respect to the trivial character, i.e.,\ $\alpha_s (\phi) = \phi$ for all $s \in G$.

\begin{lemma}\label{lemma:invariant-ideal}
    Let $\phi$ be a semi-invariant weight on $\M$. Then the following hold:
    \begin{enumerate}
        \item If there exists a nonzero $x \in \meas \M_+$ such that $\phi(x) < \infty$, then $\phi$ is semifinite.
        \item If there exists a nonzero $x \in \meas \M_+$ such that $\phi(x) \neq 0$, then $\phi$ is faithful.
    \end{enumerate}
\end{lemma}

\begin{proof}
    (i) We let $p_\phi$ be the unique projection such that $\mathcal{I} = \M p_\phi$, where $\mathcal{I}$ denotes the ultraweak closure of $\mathfrak{n}_\phi$ as in \eqref{eq:weight-ideal}. Note that since $\lim_{\epsilon \to 0} \phi(x_\epsilon) = \phi(x) < \infty$ implies that $\phi(x_\epsilon) < \infty$ for sufficiently small $\epsilon > 0$, we may assume that $x \in \M_+$. Then $x^{1/2} \in \M_+$ is a nonzero element of $\mathfrak n_\phi$, so $p_\phi \neq 0$.
    
    Observe next that $y \in \mathfrak n_\phi$ implies $\alpha_s(y) \in \mathfrak{n}_\phi$ for all $s \in G$, since
    \[
        \phi( \alpha_s(y)^* \alpha_s(y) ) 
        = \phi( \alpha_s( y^* y ) ) 
        = \chi(s)^{-1} \phi(y^* y) 
        < \infty. 
    \]
    Since $\mathcal{I} = \alpha_s(\mathcal{I}) = \alpha_s(\M p_\phi) = \M \alpha_s(p_\phi)$, the uniqueness of the projection implies that $p_\phi = \alpha_s(p_\phi)$ for all $s \in G$. Since the action is ergodic, we deduce that $p_\phi$ is a scalar multiple of the identity. As $p_\phi \neq 0$, the only possibility is that $p_\phi = 1$. Thus $\phi$ is semifinite.

    (ii) Let $\supp \phi$ be the support projection of $\phi$, that is, $\mathcal{J} = \M (1 - \supp \phi)$ where $\mathcal{J}$ is the ultraweak closure of $\mathcal{N}_\phi$ as in \eqref{eq:null_ideal}. Similarly to (i) we may assume that some nonzero $x \in \M_+$ satisfies $\phi(x) \neq 0$. Then $x^{1/2} \in \mathcal{N}_\phi$, so that $\supp \phi \neq 0$. By an argument analogous to (i) we find $\alpha_s(\mathcal{J}) = \mathcal{J}$ for all $s \in G$, so the uniqueness of $\supp \phi$ together with the ergodicity of the action implies that $\supp \phi$ is a scalar multiple of the identity. Since $\supp \phi \neq 0$, the only possibility is that $\supp \phi = 1$, so $\phi$ is faithful.
\end{proof}

The following proposition describes how the action of $G$ on $\M$ interacts with the Radon--Nikodym derivative of an associated semifinite weight. Here, the action is extended to the operators affiliated with $\M$ as in \Cref{prop:action-affiliated}.

\begin{proposition}\label{prop:semiinvariant-weights}
    Let $\phi$ be a normal, semifinite weight on $\M$ which is semi-invariant with respect to a character $\chi$. Denote by $A$ the Radon--Nikodym derivative of $\phi$. Then \[ 
        \alpha_s(A) = \chi(s) A, \quad s \in G.
    \]
    In particular, if $\phi$ is invariant, then $A = \lambda I$ for some $\lambda \geq 0$ and $\phi = \lambda \tau$.
\end{proposition}

\begin{proof}
    Let $s \in G$ and $\epsilon > 0$. Then $\alpha_s(A_\epsilon) = (\alpha_s(A))_\epsilon$ as seen in \cref{lemma:epsilon-action}. By the compatibility of the automorphism $\alpha_s$ with the functional calculus we find $\alpha_s( A_\epsilon^{1/2} ) = (\alpha_s(A))_\epsilon^{1/2}$. Using the invariance of $\tau$, for all $x \in \M_+$ we find \[
        \tau_{ \alpha_s(A)_\epsilon }(x)
        = \tau( \alpha_s(A)_\epsilon^{1/2} x \alpha_s(A)_\epsilon^{1/2} )
        = \tau( A_\epsilon^{1/2} \alpha_{s^{-1}}(x) A_\epsilon^{1/2} )
        = \tau_{A_\epsilon}( \alpha_{s^{-1}}(x) ). 
    \]
    Letting $\epsilon \to 0$, we obtain \[ 
        \tau_{\alpha_s(A)} = \alpha_s(\tau_A) = \alpha_s(\phi) = \chi(s) \phi = \chi(s) \tau_A = \tau_{\chi(s)A}. 
    \]
    The uniqueness of the Radon--Nikodym derivative $A$ then implies that $\alpha_s(A) = \chi(s) A$. 
    
    Whenever the action is ergodic and $\phi$ invariant, \Cref{prop:ergodic-affiliated} yields $A = \lambda 1$, $\lambda \geq 0$.
\end{proof}

\begin{lemma}\label{lemma:action-weight-measurable}
    Let $\phi$ be a normal, semifinite weight on $\M$. For all positive operators $x$ affiliated with $\M$, the function \[
        G \to [0,\infty], \, s \mapsto \alpha_s(\phi)(x)
    \] is measurable.
\end{lemma}
\begin{proof}
    Let $A$ denote the Radon--Nikodym derivative of $\phi$ as in \Cref{thm:radon-nikodym}. As pointwise suprema of measurable functions are measurable, it will suffice to show for all $a,b \in \M_+$ that the functions \[
        G \to [0,\infty], \, s \mapsto \alpha_s(\tau_b)(a) = \tau_a(\alpha_s(b))
    \] are measurable, as we may pick $b = A_\epsilon \in \M_+$ and $a = x_{\delta} \in \M_+$ for $\epsilon,\delta > 0$.

    By the assumption on $\alpha$, for all $b \in \M_+$ the functions $G \to \M_+, \, s \mapsto \alpha_s(b)$ are ultraweakly continuous. That means in particular, that for all $a \in \M_+$ with $\tau(a) < \infty$ the maps \[
        G \mapsto [0,\infty), \, s \mapsto \tau_a(\alpha_s(b))
    \] are continuous, and hence measurable. For more general $a \in \M_+$, as stated in \cite[p. 105]{Ne74}, we may consider an increasing net of $a_i \in \M_+$ with $\tau(a_i) < \infty$ such that $a_i$ converge to $a$ in strong operator topology. Then \[
        G \mapsto [0,\infty], \, s \mapsto \tau_a(\alpha_s(b))
    \] is a pointwise supremum of the measurable functions $s \mapsto \tau_{a_i}(\alpha_s(b))$, hence measurable.
\end{proof}

\section{Main results}

In this section we define the function $\brack{x}{y} \colon G \to \C$ for suitable operators $x$ and $y$ affiliated with $\M$ which serves as an analogue of the operator-operator convolution in Werner's quantum harmonic analysis. We then prove the main theorem from the introduction.

\subsection{Duflo--Moore for positive elements}\label{subsec:duflo-positive}

For a nonnegative measurable function $f \colon G \to [0,\infty)$ and a semifinite, normal weight $\phi$ on $\M$ the \emph{convolution} $f * \phi$ is defined to be the weight on $\M$ given by \[ 
    (f * \phi)(x) = \int_G f(s) \, \alpha_s(\phi)(x) \dif{s}, \quad x \in \M_+.
\] We stress that due to \Cref{lemma:action-weight-measurable} the integrand is a nonnegative, measurable function. By the monotone convergence theorem and the normality of $\phi$, $f * \phi$ is again a normal weight.

If $s \in G$ then we denote by $\lambda_s f \colon G \to \C$ the function given by $(\lambda_s f)(t) = f(s^{-1}t)$ for $t \in G$. By left invariance of the Haar measure we have the identity
\begin{equation}
    (\lambda_s f) * \phi = \alpha_s(f * \phi). 
    \label{eq:convolution-weight-identity}
\end{equation}

For $x,y \in \meas\M_+$ we define the function $\brack{x}{y} \colon G \to [0, \infty ]$ by \[ 
    \brack{x}{y}(s) = \alpha_s(\tau_y)(x) = \tau_{\alpha_s(y)}(x). 
\]
Since $\tau_y$ is semifinite and normal, by \Cref{lemma:action-weight-measurable} the function $\brack{x}{y}$ is measurable.

Note that for all $x,y \in \meas\M_+$ we have
\begin{equation}
    \brack{x}{y}(s^{-1}) = \brack{y}{x}(s), \quad s \in G. 
    \label{eq:bracket_symmetry}
\end{equation}
Indeed, for $x,y \in \M_+$ and $s \in G$ we consider $z = x^{1/2} \alpha_{s^{-1}}(y)^{1/2} \in \M$ and use the trace property $\tau(z^* z) = \tau(z z^*)$, which together with the invariance of the trace shows \[ 
    \brack{x}{y}(s^{-1}) 
    = \tau( x^{1/2} \alpha_{s^{-1}}(y) x^{1/2}) 
    = \tau( \alpha_s(x)^{1/2} y \alpha_s(x)^{1/2} ) 
    = \brack{y}{x}(s).
\] If $x,y \in \meas\M_+$, then taking limits of $\brack{x_\epsilon}{y_\epsilon}(s^{-1}) = \brack{y_\epsilon}{x_\epsilon}(s)$ as $\epsilon \to 0$ gives \eqref{eq:bracket_symmetry}.

We will now define admissibility of operators in the context of the group action of $G$ on $\M$. As we will see shortly, this notion coincides with the notion of admissibility with respect to an associated affiliated operator defined as in \Cref{subsec:admissibility}, cf.\ \Cref{rmk:admissibility}.

\begin{definition}\label{def:admissible-action}
    An element $y \in \M_+$ is said to be \emph{admissible} if there exists a nonzero $x \in \M_+$ such that \[ 
        \int_G \brack{x}{y} \dif{m} < \infty. 
    \]
    Furthermore, an element $y \in \M$ is \emph{admissible} if it is a linear combination of positive admissible elements.
    
    The action of $G$ on $\M$ is said to be \emph{$\tau$-integrable} if it admits a nonzero admissible element. We may always choose the nonzero admissible element to be positive.
\end{definition}

As mentioned in the introduction, we assume the ergodic action of $G$ on $\M$ to be $\tau$-integrable throughout.

\begin{theorem}\label{thm:duflo-moore}
    There exists a unique densely defined, positive, invertible operator $D$ affiliated with $\M$ such that
    \begin{equation}
        \int_G \brack{x}{y} \dif{m} = \tau(x) \tau_{D^{-1}}(y), 
        \quad x, y \in \meas\M_+. \label{eq:orthogonality-relation}
    \end{equation}
    Moreover, $D$ is semi-invariant with respect to $\Delta^{-1}$, that is, \[ 
        \alpha_s(D) = \Delta(s)^{-1} D, \quad s \in G.  
    \]
    Thus, $G$ is unimodular if and only if $D$ is a constant multiple of the identity.
\end{theorem}

\begin{proof}
    First let $x,y \in \meas\M_+$ be arbitrary and let $\mathbf{1}$ denote the function on $G$ constantly equal to $1$. As $\tau_y$ is a semifinite, normal weight, $\mathbf{1} * \tau_y$ defines a normal weight on $\M$. From \eqref{eq:convolution-weight-identity}, and using $\lambda_s\mathbf{1} = \mathbf{1}$ for all $s \in G$, we find that $\mathbf{1} * \tau_y$ is invariant. Using \eqref{eq:bracket_symmetry} we observe that
    \begin{align*}
        (\mathbf{1} * \tau_y)(x) 
        &= \int_G \alpha_s(\tau_y)(x) \dif{s} 
        = \int_G \tau_{\alpha_s(y)}(x) \dif{s} 
        = \int_G \brack{x}{y}(s) \dif{s} \\
        &= \int_G \brack{y}{x}(s^{-1}) \dif{s} 
        = \int_G \Delta(s)^{-1} \brack{y}{x}(s) \dif{s} 
        = (\Delta^{-1} * \tau_x)(y).
    \end{align*}

    Consider an admissible $y \in \M_+$. Then there exists some nonzero $x \in \M_+$ such that $({\mathbf{1} * \tau_y})(x) < \infty$, so that $\mathbf{1} * \tau_y$ is semifinite by \Cref{lemma:invariant-ideal}. It follows from \Cref{prop:semiinvariant-weights} that there exists $\lambda_y \geq 0$ such that
    \begin{equation}
        \mathbf{1} * \tau_y = \lambda_y \tau. 
        \label{eq:construction_of_lambda_for_admissible_elements}
    \end{equation}
    If $\lambda_y = 0$, then $\tau_{\alpha_s(y)}(x) = 0$ for all $x \in \M_+$ and almost every $s \in G$, which forces $y = 0$. Thus, any nonzero $x \in \M_+$ for which $\int_G \brack{x}{y} \dif{m} < \infty$ for some nonzero $y \in \M_+$ must satisfy $\tau(x) < \infty$. Let $x \in \M_+$ be such an element with $\tau(x) = 1$. As $\tau_x$ is a normal, semifinite weight, we may consider the normal weight $\phi = \Delta^{-1} * \tau_x$. Due to \eqref{eq:convolution-weight-identity}, $\phi$ is semi-invariant with respect to $\Delta$. By the property of $x$ there exists $y \in \M_+$ such that $\phi(y) = \int_G \brack{x}{y} \dif{s} < \infty$, and also $\phi(y) = \lambda_y \tau(x) = \lambda_y > 0$ since otherwise $y = 0$. By \Cref{lemma:invariant-ideal} (i) and (ii) it follows that $\phi$ is semifinite and faithful.
    
    Letting $D^{-1}$ denote the associated invertible Radon--Nikodym derivative of $\phi$ from \Cref{thm:radon-nikodym}, we obtain \begin{equation}
        \Delta^{-1} * \tau_x = \tau_{D^{-1}}.
        \label{eq:construction_of_D}
    \end{equation} We denote by $D$ the inverse of $D^{-1}$. By \Cref{prop:semiinvariant-weights} we find that $\alpha_s(D) = \Delta(s)^{-1}D$ holds for every $s \in G$. Now $G$ is unimodular if and only if $\Delta = \mathbf{1}$, in which case by \Cref{prop:ergodic-affiliated} $D$ is a constant multiple of the identity.

    Consider any nonzero $y \in \M_+$. If $y$ is not admissible, then \eqref{eq:construction_of_D} shows that $\tau_{D^{-1}}(y) = \infty$. If $y$ is admissible, and $x \in \M_+$ is as in the above paragraph, then we use \eqref{eq:construction_of_D} and \eqref{eq:construction_of_lambda_for_admissible_elements} to deduce that \[
        \tau_{D^{-1}}(y) 
        = (\Delta^{-1} * \tau_x)(y) 
        = (1 * \tau_y)(x) 
        = \lambda_y \tau(x)
        = \lambda_y.
    \] Hence \eqref{eq:construction_of_lambda_for_admissible_elements} states the following equality of weights \[
        \mathbf{1} * \tau_y = \tau_{D^{-1}}(y) \tau.
    \] This shows \eqref{eq:orthogonality-relation} for all $x,y \in \M_+$. By taking limits, \eqref{eq:orthogonality-relation} follows for all $x,y \in \meas\M_+$.
\end{proof}

\begin{definition}
    We call the affiliated operator $D$ from \Cref{thm:duflo-moore} the \emph{Duflo--Moore} operator associated to the action of $G$ on $\M$.
\end{definition}

\begin{remark}\label{rmk:admissibility}
    Notice that by \Cref{thm:duflo-moore} an element $y \in \M_+$ is admissible in the sense of \Cref{def:admissible-action} exactly when it is $D^{-1}$-admissible in the sense of \Cref{def:A-admissible}. More generally $y \in \M$ is $D^{-1}$-admissible if and only if it can be written as a linear combination of positive $D^{-1}$-admissible (hence admissible in the sense of \Cref{def:admissible-action}) elements, which by definition means that $y$ is admissible in the sense of \Cref{def:admissible-action}. We conclude that the notion of admissibility in the context of the group action $\alpha$ coincides with the notion of $D^{-1}$-admissibility.
\end{remark}

\subsection{Duflo--Moore for general elements}\label{subsec:general_operators}

We now prove versions of \Cref{thm:duflo-moore} for elements that are not necessarily positive. For $x \in L^1(\M)$ and $y \in \M$ we have $x \alpha_s(y^*) \in L^1(\M)$ for all $s \in G$. Hence we can make the following definition.

\begin{definition}\label{def:brack_general}
For $x \in L^1(\M)$ and $y \in \M$ we define the function $\brack{x}{y} \colon G \to \C$ by \[ 
    \brack{x}{y}(s) 
    = \tau(x \alpha_s(y^*)), \quad s \in G.
    \]
\end{definition}

\begin{remark}
Notice that this agrees with the previous definition for the case where both $x,y$ are positive. In fact, if $x \in L^1(\M)_+$ and $y \in \M_+$, then as $\epsilon > 0$ decreases towards $0$, $y_\epsilon \in \M_+$ increases to $y$.
Thus \[
    \lim_{\epsilon \to 0} \tau(x^{1/2} \alpha_s(y_\epsilon) x^{1/2})
    =  \tau(x^{1/2} \alpha_s(y) x^{1/2}) .
\] The trace property applied to $z(\epsilon) = x^{1/2} \alpha_s(y_\epsilon)^{1/2} \in L^2(\M)$ yields \[
    \tau_{\alpha_s(y)}(x)
    = \lim_{\epsilon \to 0} \tau(z(\epsilon)^* z(\epsilon))
    = \lim_{\epsilon \to 0} \tau(z(\epsilon) z(\epsilon)^*)
    = \tau(x^{1/2} \alpha_s(y) x^{1/2})
    = \tau(x \alpha_s(y)),
\]
where we in the final equality used that $x^{1/2} \in L^2(\M)$ and $\alpha_s(y) x^{1/2} \in L^2(\M)$.
\end{remark}

\begin{lemma}\label{lemma:bracket-pointwise-bound}
    The function $\brack{x}{y}$ is continuous for all $x \in L^1(\M)$ and $y \in \M$. Moreover, for all $a, b \in \M$ we have the following pointwise bound: \begin{equation} 
        |\brack{x}{a b^*}(s)| 
        \leq \brack{|x^*|}{a a^*}(s)^{1/2} \brack{|x|}{b b^*}(s)^{1/2}, \quad s \in G. \label{eq:pointwise_bound}
    \end{equation}
\end{lemma}

\begin{proof}
    By the assumptions on $\alpha$ the map $G \to \M$ given by $s \mapsto \alpha_{s}(y^*)$ is ultraweakly continuous. We apply this to the normal linear functional $\tau(x \,\cdot\,)$ on $\M$ determined by $x \in L^1(\M)$ to find that \[
        s \mapsto \tau(x \alpha_s(y^*)) = \brack{x}{y}(s)
    \] is a continuous function on $G$.
    
    Next, let $x = \xi \eta^*$ be any decomposition where $\xi, \eta \in L^2(\M)$. The trace property yields \[
        \brack{x}{a b^*}(s)
        = \tau( \xi \eta^* \alpha_s(a b^*)^* )
        = \tau( \eta^* \alpha_s(b) \alpha_s(a)^* \xi )
        = \tau( (\eta^* \alpha_s(b)) \cdot (\xi^* \alpha_s(a))^* ).
    \] Using the Cauchy--Schwarz inequality on $L^2(\M)$, we deduce that \[
        |\brack{x}{a b^*}(s)|
        \leq \| \eta^* \alpha_s(b) \|_2 \| \xi^* \alpha_s(a) \|_2.
    \] We compute \[
        \| \eta^* \alpha_s(b) \|_2^2
        = \tau(\alpha_s(b)^* \eta \eta^* \alpha_s(b))
        = \brack{\eta \eta^*}{ b b^* }(s).
    \] Similarly $
        \| \xi^* \alpha_s(a) \|_2^2
        = \brack{\xi \xi^*}{ a a^* }(s).
    $ 
    When $x = u |x|$ is the polar decomposition of $x$ with partial isometry $u \in \M$, we may choose $\xi = u |x|^{1/2}$ and $\eta = |x|^{1/2}$. The bound in \eqref{eq:pointwise_bound} then follows, as $\eta \eta^* = |x|$ and $\xi \xi^* = u |x| u^* = |x^*|$.
\end{proof}

\begin{proposition}\label{prop:bracket_L1_bound}
    For all $x \in L^1(\M)$ and admissible $y \in \M$ the continuous function $\brack{x}{y}$ is integrable. In particular, if $y = ab^*$ for $a,b \in \M$ that satisfy $\tau_{D^{-1}}(aa^*) < \infty$, $\tau_{D^{-1}}(bb^*) < \infty$, then \[
        \|\brack{x}{y}\|_{1} \leq \|x\|_{1} \tau_{D^{-1}}(aa^*)^{1/2} \tau_{D^{-1}}(bb^*)^{1/2}.
        \]
\end{proposition}

\begin{proof}
    Suppose $y = ab^*$ for $a,b \in \M$ that satisfy $\tau_{D^{-1}}(aa^*) < \infty$, $\tau_{D^{-1}}(bb^*) < \infty$. Using \cref{lemma:bracket-pointwise-bound}, the Cauchy--Schwarz inequality in $L^2(G)$ and \cref{thm:duflo-moore}, we obtain \begin{align*}
        \int_G | \brack{x}{a b^*}(s) | \dif{s} &\leq \int_G \brack{|x^*|}{aa^*}(s)^{1/2} \brack{|x|}{bb^*}(s)^{1/2} \dif{s} \\
        &\leq \Big( \int_G \brack{|x^*|}{aa^*}(s) \dif s \Big)^{1/2} \Big( \int_G \brack{|x|}{bb^*}(s) \dif s \Big)^{1/2} \\
        &= \tau(|x^*|)^{1/2} \tau_{D^{-1}}(aa^*)^{1/2} \tau(|x|)^{1/2} \tau_{D^{-1}}(bb^*)^{1/2} \\
        &= \|x\|_1 \tau_{D^{-1}}(aa^*)^{1/2} \tau_{D^{-1}}(bb^*)^{1/2} .
    \end{align*}
    This shows that $\brack{x}{a b^*}$ is integrable for all $a,b \in \M$ with $\tau_{D^{-1}}(a a^*) < \infty$ and $\tau_{D^{-1}}(b b^*) < \infty$.

    More generally, all admissible elements $y \in \M$ are of the form $y = \sum_{i = 1}^n a_i b_i^*$ with $a_i, b_i \in \M$ that satisfy $\tau_{D^{-1}}(a_i a_i^*) < \infty$ and $\tau_{D^{-1}}(b_i b_i^*) < \infty$. This shows that $\brack{x}{y}$ is a finite sum of integrable functions, hence integrable.
\end{proof}

\begin{lemma}\label{prop:cauchy_l1}
    Let $x \in L^1(\M)$ and $y \in \mathfrak m_\tau$. For every $\epsilon > 0$ 
    the continuous function $\brack{x}{D_\epsilon^{1/2} y D_\epsilon^{1/2}}$ is in $L^1(G)$, with 
    \begin{equation}
        \|\brack{x}{D_\epsilon^{1/2} y D_\epsilon^{1/2}}\|_{1} \leq \|x\|_{1} \|y\|_{1}. \label{eq:l1-norm-bound}
    \end{equation}
    Moreover, the net $( D^{1/2}_\epsilon y D_\epsilon^{1/2} )_{\epsilon > 0}$ converges in $L^1(G)$ as $\epsilon \to 0$.
\end{lemma}
\begin{proof}
    Consider the polar decomposition $y = u |y|$ of $y$. Set $a_\epsilon = D_\epsilon^{1/2} u |y|^{1/2} \in \mathfrak n_\tau$ and $b_\epsilon = D_\epsilon^{1/2} |y|^{1/2} \in \mathfrak n_\tau$ for every $\epsilon > 0$. Then $D_\epsilon^{1/2} y D_\epsilon^{1/2} = a_\epsilon b_\epsilon^*$.  As $D^{-1}$ and $D_\epsilon$ commute and $D^{-1} D_\epsilon = (1 + \epsilon D)^{-1} \in \M_+$, \cite[Proposition 4.3]{PeTa73} gives that $(\tau_{D^{-1}})_{D_\epsilon} = \tau_{(1 + \epsilon D)^{-1}} \leq \tau$ for all $\epsilon > 0$. Therefore \[
        \tau_{D^{-1}}(a_\epsilon a_\epsilon^*)
        = \tau_{D^{-1}}(D_\epsilon^{1/2} u |y|^{1/2} \, |y|^{1/2} u^* D_\epsilon^{1/2})
        = (\tau_{D^{-1}})_{D_\epsilon}(u |y| u^*)
        \leq \tau(u |y| u^*) = \tau(|y|) < \infty.
    \] Similarly $\tau_{D^{-1}}(b_\epsilon b_\epsilon^*) = \tau_{(1 + \epsilon D)^{-1}}(|y|) \leq \tau(|y|) < \infty$. Thus \Cref{prop:bracket_L1_bound} applies to give that $\brack{x}{D_\epsilon^{1/2} y D_\epsilon^{1/2}}$ is integrable, with
    \[
         \| \brack{x}{D_\epsilon^{1/2} y D_\epsilon^{1/2}}(s) \|_1 \leq \| x \|_1 \tau_{D^{-1}}(a_\epsilon a_\epsilon^*)^{1/2} \tau_{D^{-1}}(b_\epsilon b_\epsilon^*)^{1/2} \leq \| x \|_1 \| y \|_1 .
    \]
    Next, we show that $(\brack{x}{D_\epsilon^{1/2} y D_\epsilon^{1/2}})_{\epsilon > 0}$ forms a Cauchy net in $L^1(G)$ as $\epsilon \to 0$. Note first that \[
        \brack{x}{D_\epsilon^{1/2} y D_\epsilon^{1/2}} - \brack{x}{D_\delta^{1/2} y D_\delta^{1/2}}
        = \brack{x}{a_\epsilon (b_\epsilon - b_\delta)^*} + \brack{x}{(a_\epsilon - a_\delta) b_\delta^*}.
    \] Using the triangle inequality and \cref{prop:bracket_L1_bound} we find \[\begin{aligned}
        \|\brack{x}{D_\epsilon^{1/2} y D_\epsilon^{1/2}} - \brack{x}{D_\delta^{1/2} y D_\delta^{1/2}}\|_1
        &\leq \|x\|_1 \tau_{D^{-1}}(a_\epsilon a_\epsilon^*)^{1/2} \tau_{D^{-1}}((b_\epsilon - b_\delta) (b_\epsilon - b_\delta)^*)^{1/2} \\
        &\qquad + \|x\|_1 \tau_{D^{-1}}((a_\epsilon - a_\delta) (a_\epsilon - a_\delta)^*)^{1/2}  \tau_{D^{-1}}(b_\delta b_\delta^*)^{1/2}.
    \end{aligned}\]
    Thus, it suffices to show the following: \[
        \tau_{D^{-1}}((a_\epsilon - a_\delta) (a_\epsilon - a_\delta)^*) \to 0, 
        \quad \tau_{D^{-1}}((b_\epsilon - b_\delta) (b_\epsilon - b_\delta)^*) \to 0 
        \qquad \text{as } \epsilon, \delta \to 0.
    \]
    
    For this, we denote $d_\epsilon = (1 + \epsilon D)^{-1} \in \M_+$ for every $\epsilon > 0$. Then $D^{-1/2} D_\epsilon^{1/2} 
    = d_\epsilon^{1/2} \in \M_+$. Furthermore, from the inequality $r^{1/2} q^{1/2} \geq 2 (r^{-1} + q^{-1})^{-1}$ for real numbers $r,q > 0$ the functional calculus gives \[
        d_\epsilon^{1/2} d_\delta^{1/2}
        \geq 2 (d_\epsilon^{-1} + d_\delta^{-1})^{-1}
        = 2( 1 + \epsilon D + 1 + \delta D )^{-1}
        \geq (1 + \max\{\epsilon, \delta\} D)^{-1}
        = d_{\max\{\epsilon, \delta\}}.
    \]
    The nonnegative continuous functions $f_\epsilon(r) = (1 + \epsilon r)^{-1}$ increase pointwise and converge pointwise as $\epsilon \to 0$ to the constant $1$ function. It follows from \cref{lem:increasing_lemma} that for all $z \in \M_+$ with $\tau(z) < \infty$,\[\begin{aligned}
        0 \leq \tau( (d_\epsilon^{1/2} - d_\delta^{1/2}) z (d_\epsilon^{1/2} - d_\delta^{1/2}) )
        &= \tau_{d_\epsilon}(z) + \tau_{d_\delta}(z) - 2\tau_{d_\epsilon^{1/2} d_\delta^{1/2}}(z) \\
        &\leq \tau(z) + \tau(z) - 2\tau_{d_{\max\{\epsilon, \delta\}}}(z) \\
        &= 2( \tau(z) - \tau_{f_{\max\{\epsilon, \delta\}}(D)}(z) )
        \to 0 \quad \text{ as } \epsilon, \delta \to 0.
    \end{aligned}\]
    We apply this for $z = u |y| u^*$ and $z = |y|$ to finish the proof: \[
        \tau_{D^{-1}}( (a_\epsilon - a_\delta) (a_\epsilon - a_\delta)^* )
        = \tau( (d_\epsilon^{1/2} - d_\delta^{1/2}) u |y| u^* (d_\epsilon^{1/2} - d_\delta^{1/2}) )
        \to 0 \quad \text{ as } \epsilon, \delta \to 0,
    \] \[
        \tau_{D^{-1}}( (b_\epsilon - b_\delta) (b_\epsilon - b_\delta)^* )
        = \tau( (d_\epsilon^{1/2} - d_\delta^{1/2}) |y| (d_\epsilon^{1/2} - d_\delta^{1/2}) )
        \to 0 \quad \text{ as } \epsilon, \delta \to 0.
    \]
\end{proof}

Due to the previous lemma we may now make the following definition.

\begin{definition}\label{def:brack_D_left}
For $x \in L^1(\M)$ and $y \in \mathfrak m_\tau$ we define \[
    \brack{x}{D^{1/2} y D^{1/2}} = \lim_{\epsilon \to 0} \brack{x}{D_\epsilon^{1/2} y D_\epsilon^{1/2}}
    \]
where the limit is taken in $L^1(G)$ as in \Cref{prop:cauchy_l1}.
\end{definition}

It is immediate that $\brack{x}{D^{1/2} y D^{1/2}}$ is sesquilinear in $x \in L^1(\M)$ and $y \in \mathfrak{m}_\tau$, and that \begin{equation}
    \|\brack{x}{D^{1/2} y D^{1/2}}\|_{L^1(G)} \leq \|x\|_{L^1(\M)} \|y\|_{L^1(\M)}. \label{eq:inequality_left_D}
\end{equation}
If $y \in L^1(\M)$, one may further chose a sequence $y_n \in \mathfrak m_\tau$ with $y_n \to y$ in $L^1(\M)$, and observe that $(\brack{x}{D^{1/2} y_n D^{1/2}})_{n \in \N}$ forms a Cauchy-sequence in $L^1(G)$ which is independent of the chosen approximation. In this case we define
\begin{equation}
    \brack{x}{D^{1/2} y D^{1/2}} = \lim_{n \to \infty} \brack{x}{D^{1/2}y_nD^{1/2}} .
\end{equation}

\begin{proposition}\label{prop:duflo-general-left}
    Let $x,y \in L^1(\M)$. Then
    \begin{equation}
        \| \brack{x}{D^{1/2}yD^{1/2} } \|_1 \leq \| x \|_1 \| y \|_1  ,
    \end{equation}
    and
    \begin{equation}
        \int_G \brack{x}{D^{1/2} y D^{1/2}}(s) \dif s = \tau(x) \overline{\tau(y)}. \label{eq:integral_formula}
    \end{equation}
\end{proposition}

\begin{proof}
    Picking a sequence $(y_n)_{n \in \N}$ in $\mathfrak{m}_\tau$ such that $y_n \to y$ in $L^1(\M)$, it follows immediately from \eqref{eq:inequality_left_D} that
    \[ \| \brack{x}{D^{1/2} y D^{1/2}} \|_1 = \lim_{n \to \infty} \| \brack{x}{D^{1/2}y_nD^{1/2}} \|_1 \leq \lim_{n \to \infty} \| x \|_1 \| y_n \|_1 = \| x \|_1 \| y \|_1.  \]
    For \eqref{eq:integral_formula}, we first assume that both $x,y \in \mathfrak m_\tau$ are positive. Then from \cref{thm:duflo-moore} it follows that
    \begin{align*}
        \int_G \brack{x}{D^{1/2}yD^{1/2}}(s) \dif{s} &= \lim_{\epsilon \to 0} \int_G \brack{x}{D_\epsilon^{1/2}y D_\epsilon^{1/2}}(s) \dif{s} \\
        &= \lim_{\epsilon \to 0} \tau(x) \tau_{D^{-1}}(D_\epsilon^{1/2} y D_\epsilon^{1/2}) \\
        &= \tau(x) \tau(y) .
    \end{align*}
    More generally, we make use of the fact that all admissible $y \in \M$ can be expressed as a linear combination of four positive, admissible elements. Write $y = \sum_{k = 1}^4 \lambda_k y_k$ with $\lambda_k \in \mathbb C$ and $y_k \in \M_+$ with $\tau_{D^{-1}}(y_k) < \infty$ for $k = 1, \dots, 4$. Similarly, we may write $x = \sum_{j = 1}^4 \mu_j x_j$ with $\mu_j \in \mathbb C$ and $x_j \in \M_+$, $\tau(x_j) < \infty$ for all $j = 1,\dots, 4$. Using the sesquilinearity of $\brack{\cdot}{D^{1/2}\cdot D^{1/2}}$ and what we have already established for positive elements, we find that
    \begin{align*}
        \int_G \brack{x}{D^{1/2}yD^{1/2}} \dif m
        = \int_G \sum_{j,k = 1}^4 \mu_j \overline{\lambda_k} \, \brack{x_j}{D^{1/2}y_kD^{1/2}} \dif m
        &= \sum_{j,k = 1}^4 \mu_j \overline{\lambda_k} \tau(x_j) \tau(y_k) \\
        &= \tau(x) \overline{\tau(y)}.
    \end{align*}
    If $x \in L^1(\M)$ and $y \in \mathfrak{m}_\tau$, then we may find a sequence $(x_n)_n$ in $\mathfrak{m}_\tau$ such that $x_n \to x$ in $L^1(\M)$. In this case $\brack{x}{D^{1/2}yD^{1/2}}$ is the limit of the sequence $(\brack{x_n}{D^{1/2}yD^{1/2}})_n$ in $L^1(G)$, so \eqref{eq:integral_formula} follows from what we have already established. Finally, if $y \in L^1(\M)$, another similar approximation yields \eqref{eq:integral_formula} in this case as well.
\end{proof}

We now establish versions of \Cref{prop:duflo-general-left} with the Duflo--Moore operator on the right side of the equality. This will require $y \in \M$ to be admissible, and eventually allow it to be an element of the Banach space $L^1(\M,D^{-1})$ from \Cref{subsec:admissibility}. The following lemma ensures that the \Cref{def:brack_general} and \Cref{def:brack_D_left} do not conflict.

\begin{lemma}
    Let $x \in L^1(\M)$ and let $y \in \M$ be admissible. Then
    \[ \brack{x}{D^{1/2}(D^{-1/2}yD^{-1/2})D^{1/2}} = \brack{x}{y} \]
    as elements of $L^1(G)$.
\end{lemma}

\begin{proof}
    Since $(D^{-1})_\delta^{1/2} y (D^{-1})_\delta^{1/2} \in \mathfrak m_\tau$ converge to $D^{-1/2} y D^{-1/2}$ in $L^1(\M)$ as $\delta \to 0$, we have by definition \[\begin{aligned}
        \brack{x}{D^{1/2}(D^{-1/2}yD^{-1/2})D^{1/2}}
        &= \lim_{\delta \to 0}
        \brack{x}{D^{1/2} ((D^{-1})_\delta^{1/2} y (D^{-1})_\delta^{1/2}) D^{1/2}} \\
        &= \lim_{\delta \to 0} \lim_{\epsilon \to 0} \brack{x}{D_\epsilon^{1/2} (D^{-1})_\delta^{1/2} y (D^{-1})_\delta^{1/2} D_\epsilon^{1/2}}.
    \end{aligned}\] in $L^1(G)$. Therefore, denoting for $\epsilon, \delta > 0$ the bounded elements
    \[
        d_{\epsilon,\delta} 
        = D_\epsilon^{1/2} (D^{-1})_\delta^{1/2}
        = (1+ \epsilon D)^{-1/2} (1+ \delta D^{-1})^{-1/2} \in \M_+,
    \]
    we will have to show that \[
        \|\brack{x}{y} - \lim_{\delta \to 0} \lim_{\epsilon \to 0} \brack{x}{d_{\epsilon,\delta} y d_{\epsilon,\delta}}\|_{L^1(G)}
        = \lim_{\delta \to 0} \lim_{\epsilon \to 0} \|\brack{x}{y - d_{\epsilon,\delta} y d_{\epsilon,\delta}}\|_{L^1(G)}
        = 0.
    \]

    First, we prove the claim for $y = a b^*$ where $a, b \in \M$ satisfy $\tau_{D^{-1}}(a a^*) < \infty, \tau_{D^{-1}}(b b^*) < \infty$. We write \[
        a b^* - d_{\epsilon,\delta} a b^* d_{\epsilon,\delta}
        = (a - d_{\epsilon,\delta} a) b^* + (d_{\epsilon,\delta} a) (b - d_{\epsilon,\delta} b)^*.
    \]
    So \[
        \|\brack{x}{a b^* - d_{\epsilon,\delta} a b^* d_{\epsilon,\delta}}\|_1
        \leq \|\brack{x}{(a - d_{\epsilon,\delta} a) b^*}\|_1 + \|\brack{x}{(d_{\epsilon,\delta} a) (b - d_{\epsilon,\delta} b)^*}\|_1.
    \] Both summands are of a form were we can use the bound of \cref{prop:bracket_L1_bound}: \[
        \|\brack{x}{(a - d_{\epsilon,\delta} a) b^*}\|_1 \leq 
        \|x\|_{1} \tau_{D^{-1}}((1 - d_{\epsilon,\delta}) a a^* (1 - d_{\epsilon,\delta}))^{1/2} \tau_{D^{-1}}(bb^*)^{1/2},
    \]\[
        \|\brack{x}{ (d_{\epsilon,\delta} a) (b - b^* d_{\epsilon,\delta})^*}\|_1
        \leq \|x\|_1 \tau_{D^{-1}}(d_{\epsilon,\delta} a a^* d_{\epsilon,\delta})^{1/2} \tau_{D^{-1}}((1 - d_{\epsilon,\delta}) b b^* (1 - d_{\epsilon,\delta}))^{1/2}.
    \] We note immediately that $d_{\epsilon, \delta}^2 = (1 + \epsilon D)^{-1} (1+ \delta D^{-1})^{-1} \leq 1$, so that $\tau_{D^{-1}}(d_{\epsilon,\delta} a a^* d_{\epsilon,\delta}) \leq \tau_{D^{-1}}(a a^*)$ is bounded for all $\epsilon, \delta > 0$. So we have to show for $z \in \M_+$ with $\tau_{D^{-1}}(z) < \infty$ that \[
        \lim_{\delta \to 0} \lim_{\epsilon \to 0} (\tau_{D^{-1}})_{(1 - d_{\epsilon,\delta})^2}(z) = 0.
    \] 
    We define by $f_\epsilon(r) = (1 + \epsilon r)^{-1}$ a nonnegative, continuous function on $\mathbb R_+$ that as $\epsilon \to 0$ pointwise increases to the constant function $1$. Note that $d_{\epsilon, \delta} = f_\epsilon^{1/2}(D) f_\delta^{1/2}(D^{-1})$. We combine \cref{lem:increasing_lemma} and \cite[Proposition 4.2]{PeTa73} to compute \[
        \lim_{\delta \to 0} \lim_{\epsilon \to 0} (\tau_{D^{-1}})_{d_{\epsilon,\delta}^2}(z)
        = \lim_{\delta \to 0} \lim_{\epsilon \to 0} \big((\tau_{D^{-1}})_{f_\delta(D^{-1})}\big)_{f_\epsilon(D)}(z)
        = \lim_{\delta \to 0} (\tau_{D^{-1}})_{f_\delta(D^{-1})}(z)
        = \tau_{D^{-1}}(z),
    \]\[
        \lim_{\delta \to 0} \lim_{\epsilon \to 0} (\tau_{D^{-1}})_{d_{\epsilon,\delta}}(z)
        = \lim_{\delta \to 0} \lim_{\epsilon \to 0} \big((\tau_{D^{-1}})_{f_\delta^{1/2}(D^{-1})}\big)_{f_\epsilon^{1/2}(D)}(z)
        = \lim_{\delta \to 0} (\tau_{D^{-1}})_{f_\delta^{1/2}(D^{-1})}(z)
        = \tau_{D^{-1}}(z).
    \]
    As $
        (1 - d_{\epsilon,\delta})^2
        = 1 - 2 d_{\epsilon,\delta} + d_{\epsilon,\delta}^2,
    $ we conclude that \[
        \lim_{\delta \to 0} \lim_{\epsilon \to 0} \, (\tau_{D^{-1}})_{(1 - d_{\epsilon,\delta})^2}(z)
        = \tau_{D^{-1}}(z) - 2 \tau_{D^{-1}}(z) + \tau_{D^{-1}}(z) = 0.
    \]

    Now let $y \in \M$ be a general admissible element. We may write $y = \sum_{i = 1}^n a_i b_i^*$ where $a_i, b_i \in \M$ satisfy $\tau_{D^{-1}}(a_i a_i^*) < \infty$ and $\tau_{D^{-1}}(b_i b_i^*) < \infty$. By the triangle inequality we find \[
        \lim_{\delta \to 0} \lim_{\epsilon \to 0} \|\brack{x}{y - d_{\epsilon,\delta} y d_{\epsilon,\delta}}\|_{L^1(G)}
        \leq \sum_{i = 1}^n \lim_{\delta \to 0} \lim_{\epsilon \to 0} \|\brack{x}{a_i b_i^* - d_{\epsilon,\delta} a_i b_i^* d_{\epsilon,\delta}}\|_{L^1(G)} = 0.
    \] 
\end{proof}

\begin{definition}
Let $x \in L^1(\M)$ and let $y \in L^1(\M,D^{-1})$. We then define
\[ \brack{x}{y} = \lim_{n \to \infty} \brack{x}{D^{1/2} (D^{-1/2} y_n D^{-1/2}) D^{1/2}} , \]
where $(y_n)_{n \in \N}$ is any sequence of admissible elements in $\M$ such that $y_n \to y$ in $L^1(\M,D^{-1})$.
\end{definition}

The following corollary is now an immediate consequence of \Cref{prop:duflo-general-left}.

\begin{corollary}\label{cor:first_main_theorem}
Let $x \in L^1(\M)$ and $y \in L^1(\M,D^{-1})$. Then
\[ \| \brack{x}{y} \|_1 \leq \| x \|_1 \| D^{-1/2} y D^{-1/2} \|_1 \]
and
\[ \int_G \brack{x}{y}(s) \dif{s} = \tau(x) \overline{\tau(D^{-1/2} y D^{-1/2})}. \]
\end{corollary}

\subsection{Young's convolution inequality}

Throughout this subsection, we let $p,q,r \geq 1$ be real numbers such that
\begin{equation}
    \frac{1}{p} + \frac{1}{q} = 1 + \frac{1}{r} . \label{eq:youngs-pqr}
\end{equation}
We also denote by $p' \geq 1$ the Hölder conjugate of $p$, so that $1/p + 1/p' = 1$. We similarly denote by $q'$ the Hölder conjugate of $q$.

We recall the Araki--Lieb--Thirring inequality, cf.\ \cite{Ko92}, which states that
    \[
        \tau( (a^{1/2} b a^{1/2})^{r} ) \leq \tau( a^{r/2} b^r a^{r/2} ) , \quad a, b \in L^0(\M)_+,
    \]
provided $r \geq 1$. We use it to improve \cref{lemma:bracket-pointwise-bound} as follows:

\begin{lemma}\label{lem:pointwise_bound_pqr}
    Let $x \in L^p(\M)$ and $a, b, z \in \M$ with $\|z\|_{p'} < \infty$. 
    Then we have the following pointwise bound: \[
        |\brack{x}{a z b^*}(s)|
        \leq 
        \|x\|_{p}^{\frac p{q'}}
        \|z\|_{p'}
        \, \brack{|x^*|^p}{(a a^*)^r}(s)^{\frac1{2r}} 
        \, \brack{|x|^p}{(b b^*)^r}(s)^{\frac1{2r}},
        \quad s \in G.
    \]
\end{lemma}
\begin{proof}
    Let $x = u |x|$ be the polar decomposition with partial isometry $u \in \M$. For all $\alpha \in [0,1]$ we have $x = |x^*|^{\alpha} u |x|^{1 - \alpha}$, so as $\tau$ is a trace and $1/q' + 1/r = 1/p$ we find that \[
        \brack{x}{a z b^*}(s)
        = \tau(x \alpha_s(a z b^*)^*) = \tau(u |x|^{\frac p{q'} + \frac p{2r}} \, \alpha_s(a z b^*)^* \, |x^*|^{\frac p{2r}}), \quad s \in G.
    \] We write the expression inside $\tau$ as a product of four terms: \[
        u |x|^{\frac p{q'}} \cdot \big( |x|^{\frac p{2r}} \, \alpha_s(b) \big) \cdot \alpha_s(z)^* \cdot \big( |x^*|^{\frac p{2r}} \, \alpha_s(a) \big)^*.
    \]
    Since $1/q' + 1/(2r) + 1/p' + 1/(2r) = 1$, we can use Hölder's inequality to obtain \[
        |\brack{x}{a z b^*}(s)|
        \leq \| u |x|^{\frac p{q'}} \|_{q'} \cdot \| |x|^{\frac p{2r}} \alpha_s(b) \|_{2r} \cdot \| \alpha_s(z)^* \|_{p'} \cdot \| |x^*|^{\frac p{2r}} \, \alpha_s(a) \|_{2r}, \quad s \in G.
    \]
    Here $\| u |x|^{\frac p{q'}} \|_{q'} = \|x\|_{p}^{\frac p{q'}}$ and $\| \alpha_s(z)^* \|_{p'} = \| z \|_{p'}$ for all $s \in G$. With the Araki--Lieb--Thirring inequality we obtain \[\begin{aligned}
        \| |x|^{\frac p{2r}} \, \alpha_s(b) \|_{2r}^{2r}
        = \tau( \big( |x|^{\frac p{2r}} \, \alpha_s(b b^*) \, |x|^{\frac p{2r}} \big)^r )
        &\leq \tau( |x|^{\frac p2} \, \alpha_s(b b^*)^r \, |x|^{\frac p2} )
        = \brack{|x|^p}{ (b b^*)^r }(s).
    \end{aligned}\] Similarly, $\| |x^*|^{\frac p{2r}} \, \alpha_s(a) \|_{2r}^{2r} \leq \brack{|x^*|^p}{ (a a^*)^r }(s)$. This shows the claim.
\end{proof}

Note that for $p = q = r = 1$ we may choose $z = 1 \in \M = L^\infty(\M)$ to recover \cref{lemma:bracket-pointwise-bound}.

\begin{corollary}
    With $x$, $z$, $a$ and $b$ as in \Cref{lem:pointwise_bound_pqr}, the continuous function $\brack{x}{a z b^*}$ is an element of $L^r(G)$ whenever $\tau_{D^{-1}}( (a a^*)^r ) < \infty$ and $\tau_{D^{-1}}( (b b^*)^r ) < \infty$. Furthermore, \[
        \| \brack{x}{a z b^*} \|_{r} \leq \|x\|_{p} \|z\|_{p'} \tau_{D^{-1}}( (a a^*)^r )^{\frac 1{2r}} \tau_{D^{-1}}( (b b^*)^r )^{\frac 1{2r}}.
    \]
\end{corollary}

\begin{proof}
    This follows from the Cauchy-Schwarz inequality and \cref{thm:duflo-moore}: \[
        \int_G \brack{|x^*|^p}{(a a^*)^r} \dif m
        = \tau(|x^*|^p) \tau_{D^{-1}}( (a a^*)^r ), \;
        \int_G \brack{|x|^p}{(b b^*)^r} \dif m
        = \tau(|x|^p) \tau_{D^{-1}}( (b b^*)^r )
    \] combine to \[
        \| \brack{|x^*|^p}{(a a^*)^r}(s)^{\frac1{2}} 
        \, \brack{|x|^p}{(b b^*)^r}(s)^{\frac1{2}} \|_{L^1(G)}
        \leq \|x\|_{L^p(\M)}^p \tau_{D^{-1}}( (a a^*)^r )^{\frac1{2}} \tau_{D^{-1}}( (b b^*)^r )^{\frac1{2}}.
    \] So, using the above pointwise bound, \[\begin{aligned}
        (\int_G |\brack{x}{a z b^*}(s)|^r \dif s)^{1/r}
        &\leq 
        \|x\|_{L^p(\M)}^{\frac p{q'}}
        \|z\|_{L^{p'}(\M)}
        \| \brack{|x^*|^p}{(a a^*)^r}(s)^{\frac1{2}} 
        \, \brack{|x|^p}{(b b^*)^r}(s)^{\frac1{2}} \|_{L^1(G)}^{\frac1r} \\
        &\leq \|x\|_{L^p(\M)}^{\frac p{q'}}
        \|z\|_{L^{p'}(\M)} \|x\|_{L^p(\M)}^{\frac pr} \tau_{D^{-1}}( (a a^*)^r )^{\frac1{2r}} \tau_{D^{-1}}( (b b^*)^r )^{\frac1{2r}}.
    \end{aligned}\] As $1/q' + 1/r = 1/p$, the claim follows.
\end{proof}

\begin{proposition}\label{thm:young}
    Let $x \in L^p(\M)$, and let $y \in \mathfrak m_{\tau}$ be such that $y$ and $D$ commute. Then the net $(\brack{x}{D_\epsilon^{1/r} y})_{\epsilon > 0}$ converges in $L^r(G)$ as $\epsilon \to 0$, and \[
        \| \brack{x}{D_\epsilon^{1/r} y} \|_{r} \leq \|x\|_{p} \|y\|_{q}. 
    \]
\end{proposition}

\begin{proof}
    Fix $\epsilon > 0$. Let $y = v|y|$ be the polar decomposition of $y$. Write $D_\epsilon^{\frac 1{r}} y = D_\epsilon^{\frac 1{2r}} y D_\epsilon^{\frac 1{2r}} = a z b^*$ where $a,b,z \in \M$ are defined as follows \[
              a = D_\epsilon^{\frac 1{2r}} |y^*|^{\frac q{2r}}, 
        \quad z = |y^*|^{\frac q{2p'}} v |y|^{\frac q{2p'}}, 
        \quad b = D_\epsilon^{\frac 1{2r}} |y|^{\frac q{2r}}.
    \]
    By Hölder's inequality $\|z\|_{L^{p'}(\M)} \leq \|y\|_{L^q(\M)}^{\frac q{p'}} < \infty$. As $y$ commutes with $D$, so do $|y|,|y^*|$ and $D_\epsilon$, and hence \[
        (a a^*)^r
        = (D_\epsilon^{\frac 1{2r}} |y^*|^{\frac qr} D_\epsilon^{\frac 1{2r}})^r
        = D_\epsilon^{1/2} |y^*|^{q} D_\epsilon^{1/2}.
    \] Therefore \[
        \tau_{D^{-1}}((a a^*)^r) = \tau_{(1 + \epsilon D)^{-1}}(|y^*|^q) \leq \tau(|y^*|^q) = \|y\|_{L^q(\M)}^q < \infty.
    \] Similarly, $\tau_{D^{-1}}( (b b^*)^r ) \leq \|y\|_{L^q(\M)}^q$. This shows that $\brack{x}{D_\epsilon^{1/r} y} \in L^r(\M)$ with \[
        \|\brack{x}{D_\epsilon^{1/r} y}\|_{L^r(\M)}
        \leq \|x\|_{L^p(\M)} \|y\|_{L^q(\M)}^{\frac q{p'}} \|y\|_{L^q(\M)}^{\frac q{2r}} \|y\|_{L^q(\M)}^{\frac q{2r}}
        = \|x\|_{L^p(\M)} \|y\|_{L^q(\M)}.
    \]

    To show that $\brack{x}{D_\epsilon^{1/r} y}$ are Cauchy in $L^r(G)$, consider $\epsilon, \epsilon' > 0$ and calculate \[
        D_\epsilon^{1/r} y - D_{\epsilon'}^{1/r} y = d_{\epsilon, \epsilon'} \, y \, d_{\epsilon, \epsilon'},
        \quad \text{where } d_{\epsilon, \epsilon'} = (D_{\min\{\epsilon, \epsilon'\}}^{1/r} - D_{\max\{\epsilon, \epsilon'\}}^{1/r})^{1/2} \in \M_+.
    \] So we find that \[
        \brack{x}{D_\epsilon^{1/r} y} - \brack{x}{D_\epsilon^{1/r} y}
        = \brack{x}{(d_{\epsilon, \epsilon'} |y^*|^{\frac q{2r}}) \, z \, (d_{\epsilon, \epsilon'} |y|^{\frac q{2r}})^* }.
    \] We compute \[
        \tau_{D^{-1}}( (d_{\epsilon, \epsilon'} |y^*|^{\frac q{2r}})^{2r} )
        = \tau_{D^{-1}}( d_{\epsilon, \epsilon'}^{2r} |y^*|^q )
        = \tau_{((1 + \min\{\epsilon, \epsilon'\} D)^{-1/r} - (1 + \max\{\epsilon, \epsilon'\} D)^{-1/r})^r}(|y^*|^q).
    \] 
    We observe that $f_\epsilon(t) = (1 - (1 + \epsilon t)^{-1/r})^r$ defines non-negative, continuous functions on $\mathbb R_+$ that as $\epsilon \to 0$ pointwise decrease to $0$. We use that $(1 + \min\{\epsilon, \epsilon'\} D)^{-1/r} \leq 1 \in \M_+$ and \cref{lem:increasing_lemma} to deduce \[
        \tau_{D^{-1}}( (d_{\epsilon, \epsilon'} |y^*|^{\frac q{2r}})^{2r} )
        \leq \tau_{(1 - (1 + \max\{\epsilon, \epsilon'\} D)^{-1/r})^r}(|y^*|^q)
        \to 0 \quad \text{ as } \epsilon, \epsilon' \to 0.
    \]
    This shows that \[
        \|\brack{x}{D_\epsilon^{1/r} y} - \brack{x}{D_{\epsilon'}^{1/r} y}\|_r
        \leq \|x\|_p \|z\|_{p'} \tau_{D^{-1}}( (d_{\epsilon, \epsilon'} |y^*|^{\frac q{2r}})^{2r} )^{\frac 1{2r}} \tau_{D^{-1}}( (d_{\epsilon, \epsilon'} |y|^{\frac q{2r}})^{2r} )^{\frac 1{2r}}
        \to 0
    \] as $\epsilon, \epsilon' \to 0$.
\end{proof}

In light of \Cref{thm:young}, just like in \Cref{subsec:general_operators}, we may now define
\[ \brack{x}{D^{1/r}y} = \lim_{\epsilon \to 0} \brack{x}{D_\epsilon^{1/r}y} \]
where the limit is taken in $L^r(G)$. Afterwards, we extend the definition of $\brack{x}{D^{1/r}y}$ to the case where $x \in L^1(\M)$ and $y \in L^q(\M)$ by approximating $y$ with a sequence $(y_n)_{n \in \N}$ of elements in $\M \cap L^q(\M)$ that converges to $y$ in the $L^q(\M)$-norm. The following corollary follows immediately, establishing \cref{thm:intro2}.

\begin{corollary}\label{cor:second_main_theorem}
    If $x \in L^p(\M)$, and $y \in L^q(\M)$ commutes with $D$, then $\brack{x}{D^{1/r}y} \in L^r(\M)$, with \[
        \| \brack{x}{D^{1/r}y} \|_r \leq \| x \|_p \| y \|_q.
    \] 
\end{corollary}

Using interpolation, we can prove the boundedness on $L^p(\M)$ of the operator given by taking the bracket product with $y$ under certain assumptions. 

\begin{proposition}
    Let $y \in L^1(\M)$ be admissible and let $1 \leq p,q \leq \infty$ satisfy $1/p + 1/q = 1$. Then \[ 
        \| \brack{x}{y} \|_p \leq \| x \|_p \| y \|_1^{1/q} \| D^{-1/2} y D^{-1/2} \|_1^{1/p}, \quad x \in L^p(\M). 
    \]
\end{proposition}

\begin{proof}
    The case $p = 1$ we proved in the previous section. For $p = \infty$, $x \in L^\infty(\M) = \M$, so
    \begin{align*}
        |\brack{x}{y}(s)| 
        &= |\tau(x \alpha_s(y^*))| 
        \leq \|x\|_{L^\infty(\M)} \|\alpha_s(y^*)\|_{L^1(\M)}
        = \|x\|_\infty \|y\|_1, \quad s \in G.
    \end{align*}
    Hence the operator $x \mapsto \brack{x}{y}$ is bounded both on $L^1(\M)$ and on $L^\infty(\M)$ with norm-bounds $\| D^{-1/2} y D^{-1/2} \|_1$ and $\| y \|_1$, respectively. By the Riesz--Thorin theorem for noncommutative $L^p$-spaces (see e.g.\ \cite{DoDoPa90}) it follows that it is bounded on $L^p(\M)$ for any $1 \leq p \leq \infty$ with bound given by $\| D^{-1/2} y D^{-1/2} \|_1^{1/p} \| y \|_1^{1-1/p}$.
\end{proof}

\section{Examples}\label{sec:examples}

In this section we illustrate the main theorem through examples.

\subsection{Square-integrable unitary representations}\label{subsec:square-integrable}

Let $\pi\colon G \to \mathcal{B}(\mathcal{H})$ be a projective, irreducible, unitary representation of $G$ on a Hilbert space $\mathcal{H}$ that is \emph{square-integrable} (cf.\ \cite[Chapter 14]{Di83}), that is, there exist nonzero $\xi, \eta \in \mathcal{H}$ such that \[
    \int_G |\langle \xi, \pi_s \eta \rangle|^2 \dif{s} < \infty.
\]
Consider the von Neumann algebra $\M = \mathcal{B}(\mathcal{H})$ with the canonical tracial weight $\tau = \operatorname{Tr}$. Then \[ 
    \alpha_s(x) = \pi_s x \pi_s^*, \quad s \in G, \; x \in \M, 
\]
defines an action of $G$ on $\mathcal{B}(\mathcal{H})$ since $\pi$ is strongly continuous. It is furthermore trace-preserving (since $\pi$ is unitary) and ergodic (since $\pi$ is irreducible). The spaces $L^p(M)$ are the Schatten classes of bounded linear operators on $\mathcal{H}$ (in particular in $\M$). For positive operators $x$ and $y$ on $\mathcal{H}$ we can write \[ 
    \brack{x}{y}(s) 
    = \operatorname{Tr}(x^{1/2} \alpha_s(y) x^{1/2}) 
    = \operatorname{Tr}(x \alpha_s(y)), \quad s \in G.
\]
    
For $\eta, \zeta \in \mathcal{H}$ we denote by $\eta \otimes \zeta \in \mathcal{B}(\mathcal{H})$ the rank-one operator given by \[ 
    (\eta \otimes \zeta) \xi = \langle \xi, \eta \rangle \zeta, 
    \quad \xi \in \mathcal{H}. 
\]
It is readily shown for vectors $\xi,\xi',\eta,\eta' \in \mathcal{H}$ and all $s \in G$ that
\[\begin{aligned}
    \mathrm{Tr}( \xi \otimes \xi') 
    &= \langle \xi', \xi \rangle, \\
    \mathrm{Tr}( (\xi \otimes \xi')^* \alpha_s(\eta \otimes \eta')) 
    &= \langle \xi, \pi_s \eta \rangle \overline{ \langle \xi', \pi_s \eta' \rangle } .
\end{aligned}\]
By assumption there exist nonzero $\xi,\eta \in \mathcal{H}$ such that with $x = \xi \otimes \xi$ and $y = \eta \otimes \eta$ the condition $\int_G \mathrm{Tr}(x \alpha_s (y)) \dif{s} < \infty$ is satisfied. It follows that the action of $G$ on $\mathcal{B}(\mathcal{H})$ is $\operatorname{Tr}$-integrable. By \cref{thm:duflo-moore} there exists a densely defined, positive, invertible operator $D$ affiliated with $\mathcal{B}(\mathcal{H})$ (so we may represent it as an operator on $\mathcal{H}$) such that if $A$ is a trace class operator on $\mathcal{H}$ and $B$ is a bounded $D^{-1}$-admissible operator on $\mathcal{H}$ so that $D^{-1/2}BD^{-1/2}$ defines a trace class operator on $\mathcal{H}$, then the continuous function $s \mapsto \mathrm{Tr}(A \alpha_s(B))$ on $G$ is integrable, with \begin{equation}
    \int_G \mathrm{Tr}(A \alpha_s(B)) \dif{s} = \mathrm{Tr}(A) \mathrm{Tr}(D^{-1/2}BD^{-1/2}). \label{eq:duflo-moore-bounded-operators}
\end{equation}
This recovers \cite[Theorem 4.2, Corollary 4.4]{Ha23} which again generalizes results in \cite{BeBeLu22,LuSk18,We84}. 

Using $A = \xi \otimes \xi'$ and $B = \eta \otimes \eta'$ with $\xi,\xi' \in \mathcal{H}$ and $\eta,\eta' \in \dom(D^{-1/2})$ gives \[ 
    \int_G \langle \xi, \pi_s \eta \rangle \overline{ \langle \xi', \pi_s \eta' \rangle } \dif{s} 
    = \langle \xi, \xi' \rangle \overline{ \langle D^{-1/2} \eta, D^{-1/2} \eta' \rangle }.
\] This recovers the theorem of Duflo and Moore \cite{DuMo1976}. In particular, when $G$ is unimodular, $D = d I$ where $d$ is the formal dimension of $\pi$, which in the compact case coincides with the actual dimension of $\pi$ given that the Haar measure on $G$ is normalized to be a probability measure.

\begin{remark}\label{rmk:parity}
    For $G = \R^{2d}$ and $\mathcal{H} = L^2(\R^d)$ one can consider the projective, square-integrable, irreducible, unitary representation $\pi$ of $G$ on $\mathcal{H}$ given by \[ 
        (\pi_{(t,\omega)} f)(s) = e^{2\pi i \omega \cdot s} f(s-t), 
        \quad f \in L^2(\R^d). 
    \]
    This is the setting of Werner's work \cite{We84}. In this context the parity operator $P$ on $L^2(\R^d)$ is defined by $(Pf)(t) = f(-t)$ and one sets $\check{x} = P x P$ for a bounded linear operator $x$ on $L^2(\R^d)$. The convolution of operators $x$ and $y$ on $L^2(\R^d)$ is then defined by \[ 
        (x * y)(t,\omega) = \mathrm{Tr}(x \alpha_{(t,\omega)}(\check{y})), 
        \quad (t,\omega) \in \R^{2d}, 
    \]
    where as before $\alpha_s(y) = \pi_s y \pi_s^*$.  The relation between the convolution and bracket product of operators is given by \[
        \brack{x}{y^*} = x * \check{y}. 
    \]
    This gives a direct analogue of the convolution of two scalar-valued functions on $\R^{2d}$. Note that in e.g.\ \cite{Ha23} the convention for the convolution of operators is different from that of \cite{We84}, since it does not involve the parity operator.
\end{remark}

\subsection{Actions on measure spaces}\label{subsec:action-measure-space}

Let $(T,\mu)$ be a $\sigma$-finite measure space equipped with a measurable group action of $G$, that is, a measurable map $G \times T \to T$, $(s, t) \mapsto s t$, such that
\begin{align*}
    s (s' t) &= (s s') t, \\
    e t &= t,
\end{align*}
for $s,s' \in G$ and $t \in T$. We assume that the action is measure-preserving, that is, $\mu(s B) = \mu(B)$ for all $s \in G$ and measurable $B \subseteq T$, and ergodic, that is, whenever $s B = B$ for every $s \in G$, then $\mu(B) = 0$ or $\mu(T \setminus B) = 0$.

Consider the commutative von Neumann algebra $\M = L^\infty(T,\mu)$ with normal, semifinite, faithful tracial weight $\tau$ given by $\tau(f) = \int_T f \dif{\mu}$. The $L^p$-spaces associated with $(\M,\tau)$ are the classical $L^p$-spaces $L^p(T,\mu)$. Moreover, positive self-adjoint operators affiliated with $\M$ can be identified with a non-negative extended real-valued functions on $T$ that are finite $\mu$-almost everywhere. The action of $G$ on $T$ induces an ergodic, $\tau$-preserving action of $G$ on $\M$ given by \[ 
    \alpha_s(x)(t) = x(s^{-1} t), 
    \quad s \in G, \; x \in M, \; t \in T. 
\]
This action is $\tau$-integrable if and only if there exist measurable non-null sets $A,B \subseteq T$ such that \[ 
    \int_G \mu(A \cap s B) \dif{s} < \infty. 
\]
In this case, \Cref{thm:duflo-moore} gives the existence of a non-negative, extended real-valued, almost everywhere finite, almost everywhere nonzero function $D$ on $T$ such that \[ 
    \int_G \int_T x(t) y(s^{-1}t) \dif{\mu(t)} \dif{s} 
    = \Big( \int_T x(t) \dif{\mu(t)} \Big) \Big( \int_T y(t) D(t)^{-1} \dif{\mu(t)} \Big) 
\]
for nonnegative measurable functions $x$ and $y$ on $T$. In fact, $D^{-1}$ is the Radon--Nikodym derivative $\dif{\nu_A}/ \dif{\mu}$ where $\nu_A$ is the measure on $T$ given by $\nu_A(B) = \mu(A)^{-1} \int_G \mu(A \cap sB) \dif{s}$ for any measurable set $A \subseteq T$ with positive, finite measure.

\subsection{Convolution on a locally compact group}\label{subsec:classical-convolution}

Consider the particular case of \Cref{subsec:action-measure-space} where $G$ acts on itself by left translation. This is a measure-preserving, ergodic action with respect to the left Haar measure on $G$, and so induces a $\tau$-preserving, ergodic action of $G$ on $\M = L^\infty(G)$.

For $x,y \in L^\infty(G)_+$ we have \[ 
    \brack{x}{y}(s) = \int_G x(t) y(s^{-1}t) \dif{t}, \quad s \in G. 
\]
Hence \[ 
    \int_G \brack{x}{y} \dif{m} 
    = \int_G \int_G x(t) y(s^{-1}t) \dif{t} \dif{s} 
    = \Big( \int_G x(t) \dif{t} \Big) \Big( \int_G \Delta(s)^{-1} y(s) \dif{s} \Big), 
\]
so $y$ is admissible if and only if $\Delta^{-1} y \in L^1(G)$, that is, is integrable with respect to the right Haar measure. Picking $x$ and $y$ to be nonzero, nonnegative, compactly supported functions, the above quantity is finite, which shows that the action is $\tau$-integrable. Hence \Cref{thm:duflo-moore} applies, and the above equation tells us that the Duflo--Moore operator is given by pointwise multiplication with $\Delta$. That is,
\[ (Dx)(s) = \Delta(s) x(s) , \quad x \in L^\infty(G)_+, \quad s \in G. \]
Note furthermore that the Banach space $L^1(G,D^{-1})$ becomes the $L^1$-space of $G$ with respect to the right Haar measure.

Now for $x \in L^p(G)$ and $y \in L^q(G)$ with $1/p + 1/q = 1 + 1/r$ where $p,q,r \geq 1$, \Cref{thm:young} gives us that
\begin{equation}
   \| \brack{x}{ \Delta^{1/r} y} \|_r \leq \| x \|_p \| y \|_q \label{eq:young-commutative}
\end{equation}
where $\brack{x}{y}(s) = \int_G x(t) \overline{y(s^{-1}t)} \dif{t}$ for $s \in G$. We will reformulate this in terms of ordinary convolution of functions on $G$. For this, denote by $y^*$ the function given by $y^*(s) = \overline{y(s^{-1})}$ for $s \in G$, and note that the usual convolution of functions on $G$ is given by $(x * y)(s) = \brack{x}{y^*}$. Furthermore, $y \in L^q(G)$ if and only if $\tilde{y} = \Delta^{-1/q} y^* \in L^q(G)$ with $\| \tilde{y} \|_q = \| y \|_q$, and $\Delta^{-1/r} y^* = \Delta^{1/p'} \tilde{y}$ where $1/p + 1/p' = 1$. Hence we arrive at \[ 
    \| x * \Delta^{1/p'}\tilde{y} \|_r = \| x * \Delta^{-1/r} y^* \|_r = \| \brack{x}{\Delta^{1/r}y} \|_r \leq \| x \|_p \| \tilde{y} \|_q , \quad x \in L^p(G), \tilde{y} \in L^q(G),
\]
which is the way Young's convolution inequality is usually formulated, see e.g.\ \cite[p.\ 519]{Ro91} or \cite[Remark 2.2, p.\ 183]{KlRu78}.

\subsection{Twisted group C*-algebras of abelian groups}\label{subsec:fourier-inversion}

Let $G$ denote a locally compact group with a measurable 2-cocycle $\sigma$, that is, a measurable map $\sigma \colon G \times G \to \mathbb{T}$ such that
\begin{align*}
    \sigma(s, t) \sigma(s t, r) &= \sigma(s, t r) \sigma(t, r), \\
    \sigma(e, e) &= 1.
\end{align*}
Consider the $\sigma$-twisted group von Neumann algebra $\vN(G,\sigma)$ (cf.\ \cite{Su80}), that is, the von Neumann algebra on $L^2(G)$ generated by the operators $\lambda_\sigma(s)$ ($s \in G$) given by \[ 
    \lambda_\sigma(s) \xi(t) = \sigma(s, s^{-1}t) \xi(s^{-1}t), 
    \quad \xi \in L^2(G), \; s,t \in G. 
\]
For a function $f \in L^1(G)$ we denote by $\lambda_\sigma(f)$ the integrated form operator given by \[ 
    (\lambda_\sigma(f) \xi)(s) 
    = (f *_\sigma \xi)(s) 
    = \int_G f(t) \xi(t^{-1}s) \sigma(t,t^{-1}s) \dif{t}, 
    \quad \xi \in L^2(G). 
\]
Then $\lambda_\sigma(f) \in \vN(G,\sigma)$ and $\lambda_\sigma(f) \lambda_\sigma(f') = \lambda_\sigma(f *_\sigma f')$ for $f,f' \in L^1(G)$. Moreover $\lambda_\sigma(f)^* = \lambda_\sigma(f^*)$ where $f^*(s) = \overline{\sigma(s,s^{-1}) f(s^{-1})}$. In general $\vN(G,\sigma)$ admits a normal, semifinite, faithful weight $\phi$, called the Plancherel weight, which is determined by \[
    \phi(\lambda (f^* *_\sigma f)) 
    = \int_G | f(s)|^2 \dif{s}, 
    \quad f \in C_c(G).
\]
The Plancherel weight is tracial if and only if $G$ is unimodular, in which case it is called the Plancherel trace.

We will assume that $G$ is abelian (so that in particular it is unimodular). Denote by $\tau$ the Plancherel trace. Define an action of the Pontryagin dual $\widehat{G}$ on $\vN(G,\sigma)$ by \[ 
    \omega \lambda(f) = \lambda(\omega f), 
    \quad \omega \in \widehat{G}, \; f \in C_c(G), 
\]
where $\omega f$ denotes the pointwise product. This action is $\tau$-preserving: For $f_1,f_2 \in C_c(G)$ and $\omega \in \widehat G$ we have \[
    \omega (f_1^* *_\sigma f_2) = (\omega f_1)^* *_\sigma (\omega f_2),
\] so we may check for $f \in C_c(G)$: \[
    \tau(\omega (f^* *_\sigma f))
    = \tau((\omega f)^* *_\sigma (\omega f))
    = \int_G |\omega(s) f(s)|^2 \dif s
    = \int_G |f(s)|^2 \dif s
    = \tau(f^* *_\sigma f).
\]

Recall the Fourier transform $\mathcal F(f)(\omega) = \int_G f(s) \omega^{-1}(s) \dif s$.
Note that \[ 
    \brack{ \lambda_\sigma(f_1) }{ \lambda_\sigma(f_2) }(\omega) 
    = \tau( \lambda(f_1 *_\sigma (\omega f_2^*)) )
    = \int_G \omega(s^{-1}) f_1(s) \overline{f_2(s)} \dif s
    = \mathcal{F}(f_1 \overline{f_2})(\omega).
\]
From this it becomes apparent that the action of $\widehat G$ is $\tau$-integrable. Then all the hypotheses of \Cref{thm:duflo-moore} are satisfied, so there exists a constant $d > 0$ such that \[
    \int_{\widehat G} \brack{x}{y}(\omega) \dif \omega = d \tau(x) \tau(y^*), \quad x,y \in W^*(G,\sigma).
\]
Fix $f_1,f_2 \in C_c(G) * C_c(G)$. For $s \in G$ consider $x_s = \lambda_\sigma(s^{-1}) \lambda_\sigma(f_1)$ and $y_s = \lambda_\sigma(s^{-1}) \lambda_\sigma(f_2)$. Then $\tau(x_s) = f_1(s)$ and $\tau(y_s^*) = \overline{f_2(s)}$, while \[
    \brack{x_s}{y_s}(\omega)
    = \mathcal{F}(f_1(s \,\cdot\,) \overline{f_2(s \,\cdot\,)})(\omega)
    = \mathcal{F}(f_1 \overline{f_2})(\omega) \, \omega(s).
\] So, denoting $F = f_1 \overline{f_2}$, the Duflo--Moore theorem states that \[
    \int_{\widehat G} \mathcal{F}(F)(\omega) \, \omega(s) \dif \omega = d F(s), \quad s \in G.
\]

This is the Fourier inversion formula. Usually the Haar measure on $\widehat{G}$ is normalized according to the Haar measure on $G$ such that $d = 1$.

\subsection{Induced actions}

In this example we relate the Duflo--Moore operator of an integrable action on a von Neumann algebra of a closed subgroup to the Duflo--Moore operator of the induced action of the whole group. This gives a procedure to construct new examples to which the main theorems of the paper apply. See \cite{Ta73} for a reference of induced group actions on von Neumann algebras.

Fix as before a locally compact group $G$ with modular function $\Delta$. Let $H$ be a closed subgroup of $G$ with the property that the restriction $\Delta|_H$ is the modular function of $H$. This is equivalent to the existence of a $G$-invariant Radon measure $\mu$ on the quotient $G/H$, cf.\ \cite[Theorem 2.49]{Fo95}, which we normalize such that Weil's integration formula\begin{equation}\label{eq:quotient_integral_formula}
    \int_G f(s) \dif s = \int_{G/H} \int_H f(s t) \dif t \dif {\mu(sH)}
\end{equation}
holds for all nonnegative Borel functions $f$ on $G$.

Let $\mathcal N$ be a von Neumann algebra with normal, semifinite, faithful trace $\kappa$ and suppose that $\mathcal N$ is equipped with an action $\beta$ of $H$. Suppose that $\beta$ is ergodic, $\kappa$-preserving and $\kappa$-integrable. 
Consider the von Neumann algebra \[ 
    \M = \operatorname{Ind}_H^G(\mathcal N) 
    = \{x \in L^\infty(G, \mathcal N) : x(s t^{-1}) = \beta_t(x(s)) \text{ for all } s \in G \text{ and } t \in H \}.
\]
An action $\alpha$ of $G$ on $\M$ is given by \[ 
    \alpha_s(x)(s') = x(s^{-1} s'), \quad s,s' \in G. 
\]
This is the \emph{induced action} of $\beta$ on $\mathcal N$, which is easily checked to be ergodic.

As $\beta$ is $\kappa$-preserving, for all $x \in \M_+$ the function $s \mapsto \kappa(x(s))$ is right $H$-invariant, hence can be considered a function on $G/H$. We define \[
    \tau(x) = \int_{G/H} \kappa(x(s)) \dif \mu(sH), \quad x \in \M_+.
\] One checks that $\tau$ defines a normal, semifinite, faithful trace on $\M$, and moreover that $\alpha$ is $\tau$-preserving.

We denote by $C$ the Duflo--Moore operator corresponding to the action $\beta$ of $H$ on $\mathcal N$. As $C$ is semi-invariant with respect to $\Delta^{-1}$, for every $y \in \M_+$ the function $s \mapsto \Delta^{-1}(s) \kappa_{C^{-1}}(y(s))$ is right $H$-invariant, hence can be considered a function on $G/H$. Assuming that $\alpha$ is $\tau$-integrable, we claim that \begin{equation}\label{eq:duflo-moore-induced-action}
    \tau_{D^{-1}}(y) = \int_{G/H} \Delta^{-1}(s) \, \kappa_{C^{-1}}(y(s)) \dif \mu(sH), \quad y \in \M_+,
\end{equation} determines the Duflo--Moore operator $D$ for the action $\alpha$ of $G$ on $\M$. We remark that this can be seen as an extension of \cref{subsec:classical-convolution} by choosing the trivial subgroup.

Given $w \in \mathcal N_+$ with $\kappa(w) < \infty$ and $y \in \M_+$ we consider the nonnegative function $f$ on $G$ defined by $f(s) = \kappa(w y(s))$ for $s \in G$. Then \[
    f(s t^{-1}) 
    = \kappa(w \beta_t(y(s))) 
    = \,_\beta\langle w, y(s) \rangle(t),
    \quad s \in G, t \in H.
\] By the Duflo--Moore theorem for $\beta$ we find \[
    \int_H f(s t^{-1}) \dif t
    = \int_H \,_\beta\langle w, y(s) \rangle(t) \dif t
    = \kappa(w) \kappa_{C^{-1}}(y(s)).
\] We note that in general \eqref{eq:quotient_integral_formula} may be applied to functions $\Delta^{-1} f$ to yield \[
    \int_G f(s^{-1}) \dif s
    = \int_{G/H} \Delta^{-1}(s) \int_H f(s t^{-1}) \dif t \dif \mu(s H).
\] This shows \[\begin{aligned}
    \int_G \kappa(w y(s^{-1})) \dif s
    &= \int_{G/H} \Delta^{-1}(s) \, \kappa(w) \kappa_{C^{-1}}(y(s)) \dif \mu(s H).
\end{aligned}\]
Let now $x,y \in \M_+$ with $\tau(x) < \infty$. The above computation may be carried out for all $w = x(s') \in \mathcal N_+, s' \in G$. Then in combination with Fubini's theorem and the left-invariance of the Haar measure we conclude \[\begin{aligned}
    \int_G \tau(x \alpha_s(y)) \dif s
    &= \int_G \int_{G/H} \kappa(x(s') y(s^{-1} s')) \dif \mu(s'H) \dif s \\
    &= \int_{G/H} \int_G \kappa(x(s') y(s^{-1})) \dif s \dif \mu(s'H) \\
    &= \int_{G/H} \int_{G/H} \Delta^{-1}(s) \, \kappa(x(s')) \kappa_{C^{-1}}(y(s)) \dif \mu(s H) \dif \mu(s'H) \\
    &= \tau(x) \int_{G/H} \Delta^{-1}(s) \, \kappa_{C^{-1}}(y(s)) \dif \mu(s H).
\end{aligned}\]
As $x,y \in \M_+$ were arbitrary, this shows \eqref{eq:duflo-moore-induced-action}.

\printbibliography

\end{document}